\documentclass[11pt]{article}
\usepackage{etoolbox}
\usepackage{rtrssn}

\usepackage[square,sort,comma,numbers]{natbib}

\def\re#1{{#1}}
\begin{document}

\title{Riemannian Trust Region Methods for SC$^1$ Minimization}

% \author{Chenyu Zhang\and Rufeng Xiao\and Wen Huang\and Rujun Jiang}
%
% %\authorrunning{Short form of author list} % if too long for running head
%
% \institute{Chenyu Zhang \at
% 	Data Science Institute, Columbia University, New York, USA.\\
% 	\email{chenyu.zhang@columbia.edu}
% 	\and
% 	Rufeng Xiao \at
% 	School of Data Science, Fudan University, Shanghai, China.\\
% 	\email{rfxiao21@m.fudan.edu.cn}
% 	\and
% 	Wen Huang \at
% 	School of Mathematical Sciences, Xiamen University, Xiamen, China.\\
% 	\email{wen.huang@xmu.edu.cn}
% 	\and
% 	Rujun Jiang \at
% 	School of Data Science, Fudan University, Shanghai, China.\\
% 	\email{rjjiang@fudan.edu.cn}
% }
%
% \date{Received: date / Accepted: date}

\author{Chenyu Zhang\thanks{Data Science Institute, Columbia University, New York, USA. (chenyu.zhang@columbia.edu).}\and
Rufeng Xiao\thanks{School of Data Science, Fudan University, Shanghai, China. (rfxiao21@m.fudan.edu.cn).}\and
Wen Huang\thanks{School of Mathematical Sciences, Xiamen University, Xiamen, China. (wen.huang@xmu.edu.cn).}\and
Rujun Jiang\thanks{School of Data Science, Fudan University, Shanghai, China. (rjjiang@fudan.edu.cn).}}
\date{}

\maketitle

\begin{abstract}
	Manifold optimization has recently gained significant attention due to its wide range of applications in various areas. This paper introduces the first Riemannian trust region method for minimizing an SC$^1$ function, which is a differentiable function that has a \textit{semismooth} gradient vector field, on manifolds with convergence guarantee. We provide proof of both global and local convergence results, along with demonstrating the local superlinear convergence rate of our proposed method. As an application and to demonstrate our motivation, we utilize our trust region method as a subproblem solver within an augmented Lagrangian method for minimizing nonsmooth nonconvex functions over manifolds. This represents the first approach that fully explores the second-order information of the subproblem in the context of augmented Lagrangian methods on manifolds. Numerical experiments confirm that our method outperforms existing methods.
	% \keywords{manifold optimization \and Riemannian trust region methods \and semismooth}
	% \subclass{90C30\and 49J52\and 65K05\and 90C26}
\end{abstract}

\section{Introduction}
Manifold optimization has emerged as a significant research area due to its broad applicability in various fields, including phase retrieval \Citep{cai2019fast, bendory2017non}, phase synchronization \Citep{boumal2016nonconvex,liu2017estimation}, low-rank matrix completion \Citep{boumal2011rtrmc, vandereycken2013low}, principal component analysis \Citep{lu2012augmented,montanari2015non}, and deep learning \Citep{cho2017riemannian}.
In a manifold optimization problem, the feasible region is on a smooth manifold, such as a sphere or a Stiefel manifold.
Extensive research has been conducted in the past few decades on optimizing smooth objective functions on manifolds
\Citep{gabay1982minimizing,hu2020brief,hu2018adaptive,wen2013feasible,zhang2016riemannian}, and \cite{absilOptimizationAlgorithmsMatrix2008} summarizes several classical algorithms in this field, such as Newton's method, line-search methods, and trust region methods.
However, these methods encounter challenges when the objective function becomes nonconvex. Hence, manifold optimization with a nonconvex objective function has become an active area of research in recent years \Citep{absilTrustRegionMethodsRiemannian2007,chen2020proximal,chen2016augmented,huang2021riemannian,lai2014splitting}.

In this paper, we consider an unconstrained nonconvex optimization problem on a manifold:
\begin{equation}\label{eq:problem}
  \min_{x\in\M}\phi(x)
\end{equation}
where $\varphi$ is bounded below on a complete Riemannian manifold $\mathcal{M}$, has a Lipschitz continuous and locally directionally differentiable gradient field, but may not be twice differentiable.
Particularly, we will consider the case where $\phi$ is an SC$^{1}$-function: a differentiable function that has a \textit{semismooth} gradient vector field. We defer the definition of semismoothness to \cref{sec:alg}.
SC$^1$ objective functions are commonly encountered in various domains, including stochastic quadratic programs \citep{qi1995sqp} and nonlinear minimax problems \citep{pang1995globally}.

\subsection{Related Work}

% Related work on semismooth Newton methods
\paragraph{Semismooth Newton methods.}
For an SC$^1$ problem~\cref{eq:problem} in a Euclidean space, semismooth Newton (SSN) methods have been widely applied \cite{kummer1988Newtonmethod,qi1993nonsmoothversion}.
Recently, \cite{deoliveiraNewtonMethodFinding2020} extended the SSN methods to Riemannian manifolds.
Then, the Riemannian semismooth Newton method was applied to solve various optimization problems on manifolds. For instance,
\cite{diepeveen2021inexact} applied it to solve a primal-dual optimality system on a manifold,
while \cite{zhou2022semismooth} applied it to solve the subproblem of ALM on manifolds.
In both papers, the Newton system was solved inexactly, and a superlinear local convergence rate was established.

% Related work on semismooth trust region methods
\paragraph{Trust region methods.}
\cite{absilTrustRegionMethodsRiemannian2007} extended trust region methods to Riemannian manifolds, and established a superlinear convergence result similar to the Euclidean case. However, their smoothness requirements for the objective function are strong, assuming $\phi$ is twice continuously differentiable and its Hessian is Lipschitz continuous.
For objective functions that may not be twice differentiable, generalized Hessians have been considered for Euclidean trust region methods.
\cite{sun1997computablegeneralized} proposed a globally
and superlinearly convergent trust region algorithm for the variational inequality problem, which utilizes the D-gap function and its computable generalized Hessian.
For the same problem, \cite{yang2000trustregion} also proposed a trust region type method, which only switches to a trust region step when Newton's step fails to yield a sufficient decrease, thus avoiding the use of the trust region near a strict local minimum.
\cite{jiang1996Globallysuperlinearly} considered a constrained convex SC$^{1}$ problem, and similar to \cite{yang2000trustregion}, their algorithm only resorts to a trust region strategy when Newton's step fails.

\paragraph{Problem \eqref{eq:problem} as subproblems in two methods.}
Recently, \cite{kangkang2022inexact} and \cite{zhou2022semismooth} extended augmented Lagrangian methods to nonsmooth nonconvex manifold optimization and established convergence guarantees, which motivates the research in this paper. However, neither of these papers incorporates a full second-order method to solve the subproblem in the ALM; the former solves the subproblem using the Riemannian gradient descent method, while the latter employs an SSN method that falls back to the gradient descent method when encountering negative curvatures.
For minimizing a composite function over a manifold, \cite{beck2023dynamic} reformulated the objective function utilizing dynamic smoothing, whose subproblem is also in form \eqref{eq:problem}.

\subsection{Contributions}
In this paper, we introduce a novel Riemannian trust region method for minimizing SC$^1$ functions on Riemannian manifolds. We provide empirical evidence of the global convergence, local convergence near nondegenerate local minima, and superlinear local convergence rate of our method under mild conditions, adapted from its Euclidean counterpart. Our method represents is the first Riemannian trust region approach to attain a provable superlinear local convergence rate without the need for the objective function to be twice differentiable. Moreover, we relax the smoothness requirement on the retraction to only necessitate a Hölder continuous differential, which aligns with the SC$^{1}$ objective function. In contrast, the prior work \cite{absilTrustRegionMethodsRiemannian2007}  requires the retraction to have a Lipschitz continuous differential for global convergence of the algorithm and to be twice continuously differentiable for the algorithm's superlinear local convergence rate. Furthermore, we provide a proof of the trust region's inactivity near a nondegenerate minimizer, which plays a crucial role in establishing the superlinear local convergence rate. To the best of our knowledge, this is the first guarantee of the eventual inactivity of the trust region for semismooth trust region methods, even within the framework of Euclidean spaces.

As an important application, we employ our semismooth Riemannian trust region method to solve the subproblem of the ALM on manifolds. Notably, our approach stands out as the first method to fully exploit the second-order information of the ALM's subproblem, thereby benefiting from the advantages offered by trust region methods over first-order methods and the Newton method, including adaptive step-size and automatic detection of negative curvature.
Through numerical experiments on compressed models and sparse principal component analysis, we demonstrate that our proposed method outperforms existing methods, achieving better convergence performance, characterized by faster convergence speed and improved objective function values.

\subsection{Organization}
This paper is organized as follows. In \cref{sec:pre}, we briefly review the basic concepts in Riemannian manifold optimization and trust region methods, introducing tools required for our algorithms and analysis.
In \cref{sec:alg}, we present our Riemannian trust region method for minimizing an SC$^1$ function on a manifold.
In \cref{sec:anlys}, we establish the convergence results of our algorithm, including global convergence, local convergence (including attraction property of nondegenerate local minimizers), and superlinear local convergence rate. In \cref{sec:app}, we present an application of our method, solving subproblems of augmented Lagrangian methods on manifolds. Finally, we evaluate our algorithm through multiple numerical experiments, including compressed modes and sparse principal component analysis in \cref{sec:exp}.

\section{Preliminaries} \label{sec:pre}
\subsection{Riemannian Manifold Optimization}

In this section, we provide a brief introduction to Riemannian manifold optimization, assuming familiarity with basic concepts such as smooth manifolds, tangent spaces, and smooth mappings on manifolds.
Omitted definitions in this section and more details can be found in monographs such as \citep{boothbyIntroductionDifferentiableManifolds2003} and \citep{absilOptimizationAlgorithmsMatrix2008}.
We summarize the notations on Riemannian manifold optimization we will use in \cref{tb:notation}.

\begin{table}[ht]
  \caption{~Notations}
  \label{tb:notation}
  \centering
  \begin{tabular}{cl}
    \toprule
    Notation                                        & Definition                                                                         \\
    \midrule
    $\M$                                            & A complete Riemannian manifold                                                     \\
    $T_x\M$                                         & The tangent space at $x\in\M$                                                      \\
    $\mathcal{L}(T_x\M)$                            & The set of linear operators from $\tm{x}$ to itself                                \\
    $\R,\NN$                                        &
    The set of all real numbers; the set of all natural numbers
    \\
    $\R^{k},E$                                      & A vector space; a general Euclidean space                                          \\
    $\R^{k}_{-},\R^{k}_{+}$                         & The set of elements in $\R^{k}$ with non-positive/non-negative components          \\
    $B_{\delta}(x)$                                 & The ball with center $p\in S$ and radius $\delta$,                                 \\[-1ex]&where $p \in S$ can be $x\in\M$, $0_{x}\in\tm{x}$, $\operatorname{id}_{\tm{x}}\in \mathcal{L}(\tm{x})$, etc.\\
    $x,y,z,p,q$                                     & Points on manifolds                                                                \\
    $X,Y,Z$                                         & Vector fields on manifolds                                                         \\
    $\xi,\eta,\zeta,v$                              & Vectors in the tangent space                                                       \\
    $\gamma$                                        & A piecewise smooth curve on manifolds                                              \\
    $\lang\cdot,\cdot\rang_x,\lang\cdot,\cdot\rang$ & The Riemannian inner product on $\tm{x}$                                           \\
    $\|\cdot\|_x,\|\cdot\|$                         & The Riemannian norm on $\tm{x}$                                                    \\
    $\|\cdot\|_{\mathrm{op}}$                       & The operator norm                                                                  \\
    $\dist(x,y)$                                    & The distance between $x,y\in\M$                                                    \\
    $\d f$                                          & The differential of function $f$ on manifolds                                      \\
    $\nabla_{v} X$                                  & The Riemannian covariant derivative of $X$ at $t$ along $v$ ($\gamma'(t) = v$)     \\
    $\nabla$                                        & The Riemannian connection (Levi-Civita connection) on $\M$                         \\
    $P_\gamma^{t_1\to t_2}$                         & The parallel transport along $\gamma$ from $t_1$ to $t_2$                          \\
    $P_{xy}$                                        & The parallel transport along the geodesic connecting from $x$ to $y$               \\
    $\exp_x$                                        & The exponential map at $x\in\M$                                                    \\
    $R_x$                                           & The retraction restrict to $\tm{x}$                                                \\
    $\grad f$                                       & The Riemannian gradient of $f$                                                     \\
    $\H f$                                          & The Riemannian Hessian of $f$                                                      \\
    $\pat f,\pat X$                                 & The Clarke subdifferential of function $f$;                                        \\[-1ex]& the Clarke generalized covariant derivative of the vector field $X$\\
    $H_{x_k},H_{k}$                                 & An element in the Clarke generalized covariant derivative $\pat\!\grad\phi(x_{k})$
    \\\bottomrule
  \end{tabular}
\end{table}

In this paper, we focus on Riemannian manifolds, which are smooth manifolds equipped with an inner product $\left<\cdot,\cdot \right>_x$ on the tangent space $\tm{x}$ varying smoothly with respect to $x$ on the manifold $\M$. The family of inner products is called the Riemannian metric on the manifold.
This paper deals with general Riemannian manifolds, and we only consider their intrinsic properties and always use the notation $\left<\cdot, \cdot \right>_{x}$ to refer to the Riemannian inner product.

The Riemannian metric also introduces a norm on the tangent space, defined by $\|\xi\|_{x} = \sqrt{\left<\xi,\xi \right>_{x}}$ for any $\xi\in\tm{x}$, and a distance on the manifold, defined by $\dist(x,y) = \inf_{\gamma}\int_{0}^{1}\|\gamma'(t)\|_{\gamma(t)}\d t$ for all $x,y\in\M$, where the infimum is taken over all piecewise smooth curves $\gamma: [0,1] \to \M$ connecting $x$ and $y$.
Throughout the paper, we will drop the subscript $x$ of the Riemannian inner product and norm if it is clear from the context.

In manifold optimization, we still need to return to linear spaces, like tangent spaces, to perform various operations.
However, unlike the Euclidean case, the tangent vectors at different points of a manifold are not in the same tangent space. So we need a mapping to bridge different tangent spaces. This is where the concepts of geodesics and parallel transports come into play.

\begin{definition}[Parallel, geodesic, and parallel transport]\label{def:geo-para}
	Let $\gamma:[0,1]\to \M$ be a smooth curve.
	\begin{itemize}
		\item A vector field $X$ is said to be parallel along $\gamma$ if $\nabla_{\gamma'(t)}X = 0$ for any $t \in [0,1]$, where $\nabla$ is the Riemannian connection (Levi-Civita connection).
		\item $\gamma$ is said to be geodesic if the field of its tangent vector $\gamma'(t)$ is parallel along itself.
		\item For any $\xi \in \tm{\gamma(0)}$, there exists a unique parallel vector field $X_\xi$ along $\gamma$ such that $X_\xi(0) = \xi$.
		      The parallel transport operator along $\gamma$ is defined by $P_{\gamma}^{0\to t}: \xi \mapsto X_\xi(t)$. When the curve is geodesic and connects $x,y$, we denote $P_{xy} \coloneqq P_{\gamma}^{0\to 1}$.
	\end{itemize}
\end{definition}

The parallel transport is an isometry that preserves the inner product, i.e., $\left<P_{xy}\xi_{x}, P_{xy}\eta_{x} \right>_{y} = \left<\xi_{x},\eta_{x} \right>_{x}$ for any tangent vectors $\xi_{x},\eta_{x}\in\tm{x}$. Hence, we can freely transfer vectors between different tangent spaces using the parallel transport.
From the tangent space to the manifold, geodesics naturally introduce a local map called the exponential map defined as $\exp_{x}: \xi \mapsto \gamma(1)$, where $\gamma$ is a unique geodesic such that $\gamma(0) = x, \gamma'(0) = \xi$, and $\dist(\gamma(0),\gamma(1)) = \|\xi\|$. The exponential map is a local diffeomorphism, and we define the injective radius of $x$ as follows:
$$
	\operatorname{inj}_{x}(\M) \coloneqq \sup \{ \delta>0: \exp_{x} \text{ is a diffeomorphism on } B_{\delta}(0_{x}) \subset \tm{x} \}.
$$
The global injective radius of the manifold is defined as $\operatorname{inj}(\M) \coloneqq \inf_{x\in\M}\operatorname{inj}_{x}(\M)$.
\re{
	The \textit{geodesically convex normal neighborhood} (or \textit{normal neighborhood} for short) of $x$ is a neighborhood $\operatorname{NN}(x)$ such that for any $y,z\in  \operatorname{NN}(x)$, there is a unique minimizing geodesic segment from $y$ to $z$ in $\M$, and the image of this geodesic segment lies entirely in $\operatorname{NN}(x)$.
	There is a nonempty geodesically convex normal neighborhood at each point of a Riemannian manifold \cite[Theorem 6.17]{lee2012Smoothmanifolds}.
}
We rely heavily on the smooth diffeomorphism offered by the exponential map, both in the algorithms and analysis. Therefore, throughout the paper, we restrict all neighborhoods of a point $x$ to be within $x$'s \re{geodesically convex normal neighborhood}.
This restriction is possible when we only consider a compact subset $\Omega$ of the manifold. Thus the injective radius of $\Omega$ has a positive lower bound \re{\cite{huang2021riemannian,diepeveen2021inexact}},
\re{and $\exp_{x}(B_{\frac{1}{2}\operatorname{inj}(\Omega)}(0_{x})) \subset \operatorname{NN}(x)$ for any $x\in\M$.
	For example, any point on a unit sphere with Euclidean metric has an injective radius of $\pi$ and a geodesically convex normal neighborhood $\operatorname{int}B_{\pi /2}$.
}
% For the rest of the paper, we omit this requirement and the concept of injective radius \re{and normal neighborhood}.
% \re{\sout{%
% 	Restricting our discussion to the normal neighborhood allows us to drop the completeness assumption on the Riemannian metric, which is crucial for the convergence analysis of many manifold optimization methods \cite{absilTrustRegionMethodsRiemannian2007,zhou2022semismooth,deoliveiraNewtonMethodFinding2020,dedieu2003Newtonmethod}.
% 	The normal neighborhood ensures the local well-definedness of the exponential map and geodesics, replacing the global well-definedness provided by the completeness assumption and the Hopf-Rinow theorem.
% 	Thus, our method applies to incomplete Riemannian manifolds, such as fixed-rank matrix manifolds with the Euclidean metric \cite{meyer2011regression,vandereycken2013low}.
% }}
In addition, we assume that the manifold is complete, i.e., any locally defined geodesic can be extended to the entire real axis. This assumption ensures that the exponential map is well-defined on the entire tangent space, and the shortest curve connecting two points is a geodesic (see \cite{carmoDifferentialGeometryCurves2016}).

We now introduce some key tools related to functions on manifolds that will help us apply algorithms to solve manifold optimization problems.

\begin{definition}[Riemannian gradient and Hessian]
	For a function from $\M$ to $\R$, its Riemannian gradient at $x$ is a unique tangent vector such that
	$$
		\lang \grad f(x), \xi\rang = \d f(x)[\xi],\quad \forall \xi \in \tm{x},
	$$
	where $\d f(x)$ is the differential of $f$ at $x$; its Riemannian Hessian is an element in $\mathcal{L}(\tm{x})$, the set of all linear operators from $\tm{x}$ to itself, such that
	$$
		\H f(x)[\xi] = \nabla_{\xi}\grad f(x),\quad \forall \xi \in \tm{x},
	$$
	where $\nabla$ is the Riemannian connection (Levi-Civita connection) on $\M$.
\end{definition}

The problem we consider is not necessarily smooth; to obtain its second-order information, we need the following definitions, which can be found in \Citep{deoliveiraNewtonMethodFinding2020,ghahraei2017PseudoJacobiancharacterization}.

\begin{definition}[Lipschitz continuous vector field]\label{def:lip}
	A vector field $X$ on $\M$ is said to be locally Lipschitz continuous if for any $x\in\M$, there exist a radius $\delta_x > 0$ and a constant $L_x >0$ such that
	$$
		\|P_{yz}X(y) - X(z)\| \le L_x \dist (y,z), \quad \forall y,z \in B_{\delta_x}(x).
	$$
	A function is said to be Lipschitz continuously differentiable if its gradient vector field is a Lipschitz vector field with a global Lipschitz constant $L$ and a global radius of neighborhood $\delta$.
\end{definition}

\begin{definition}[Directional derivative]
	For a vector field $X$, its directional derivative at $x\in \M$ along $\xi\in\tm{x}$ is defined as
	$$
		\nabla X(x;\xi)\coloneqq \lim_{t\to 0^+}\frac{1}{t}\left( P_{\exp_{x}(t\xi),x} X(\exp_{x}(t\xi)) - X(x) \right).
	$$
	$X$ is said to be directionally differentiable at $x$ if $\nabla X(x;\xi)$ exists for all $\xi \in \tm{x}$.
\end{definition}

If $X$ is differentiable at $x$, then it is directionally differentiable, and $\nabla X(x;\xi) = \nabla_\xi X(x)$ for all $\xi \in \tm{x}$.

\iffalse

\begin{definition}[Clarke subdifferential]
	For a locally Lipschitz continuous function $f$ on $\M$, its Clarke Riemannian generalized directional derivative at $x\in\M$ along $v\in\tm{x}$ is defined as
	$$
		f^{\circ}(x;v) \coloneqq \lim_{y \to x}\sup_{t\to 0+}\frac{f\circ \phi^{-1}(\phi(y) + t\d \phi(x)[v]) - f\circ \phi^{-1}(\phi(y))}{t},
	$$
	where $(U,\phi)$ is any coordinate chart at $x$.
	Then the Clarke subdifferential of $f$ at $x\in\M$ is defined as
	$$
		\partial f(x) \coloneqq \{ \xi\in\tm{x}:\left<\xi,v \right>\le f^{\circ}(x;v), \forall v \in\tm{x} \}.
	$$
\end{definition}

The elements in the Clarke subdifferential can be viewed as generalized gradients and are frequently used in subgradient methods for manifold optimization \citep{grohs2016nonsmooth,chen2020proximal}.
\fi
In subgradient methods for manifold optimization,
the elements in the Clarke subdifferential are frequently used as generalized gradients \citep{grohs2016nonsmooth,chen2020proximal}.
Similarly, we can define a generalized Hessian for non-twice differentiable functions.
\citep[Theorem 3.2]{deoliveiraNewtonMethodFinding2020} states that locally Lipschitz continuous vector fields on $\M$ are differentiable almost everywhere, allowing us to introduce the Clarke generalized covariant derivative of such vector fields.

\begin{definition}[Clarke generalized covariant derivative] \label{def:cgcd}
	The Clarke generalized covariant derivative of a locally Lipschitz continuous vector field $X$ at $x\in\M$ is defined as
	$$
		\pat X(x) \coloneqq \conv \left\{
		\begin{aligned}
			H\in \mathcal{L}(\tm{x}) : \exists \{x_k\} \subset \mathcal{D}_X, x & = \lim_{k\to\infty}x_k,                             \\
			H                                                                   & = \lim_{k\to \infty}P_{x_kx}\nabla X(x_k) P_{x x_k}
		\end{aligned}\right\},
	$$
	where $\mathcal{D}_X$ is the collection of points on the manifold where $X$ is differentiable.
	Since Hessians at differentiable points are self-adjoint \cite[Lemma 11.1]{lee2018IntroductionRiemannian}, all elements in the Clarke generalized covariant derivative, and their inverse (if exists), are self-adjoint.
\end{definition}

\iffalse
Note that we use the same notation for the Clarke subdifferential and the Clarke generalized covariant derivative, as these two definitions are equivalent and incorporate the Euclidean case (see \cite[Theorem 2.5.1]{clarkeOptimizationNonsmoothAnalysis1990}),
although we can define other generalized gradients, such as the Fréchet subdifferential, and establish their corresponding optimality conditions and related algorithms \citep{ledyaevNonsmoothAnalysisSmooth2007}.
\fi
While elements in the Clarke generalized covariant derivative of the gradient field are not necessarily Hessian operators, they possess desired properties that make them suitable replacements for Hessian operators in our algorithms.

\begin{proposition}[{\citep[Proposition 3.1]{deoliveiraNewtonMethodFinding2020}}] \label{lem:cgcd}
	Let $\pat X$ be the Clarke generalized covariant derivative of a locally Lipschitz continuous vector field $X$ on $\M$. The following statements are valid for any $x \in \M$:
	\begin{enumerate}
		\item $\partial X(x)$ is a nonempty, convex, and compact subset of $\mathcal{L}\left(T_{x}\M\right)$;
		\item $\pat X$ is locally bounded; that is, for any $\delta>0$, there exists $C>0$ such that for all $y \in B_{\delta}(x)$ and $H \in \partial X(y)$, it holds that $\|H\| \leq C$;
		\item $\partial X$ is upper semicontinuous at $p$; that is, for any scalar $\epsilon>0$, there exists $\delta > 0$ such that for all $y \in B_{\delta}(x)$, it holds that
		      $$
			      P_{yx} \partial X(y) P_{xy}\subset \partial X(x)+B_{\epsilon}(0),
		      $$
		      where $B_{\epsilon}(0):=\left\{H \in \mathcal{L}\left(T_{x} \M\right):\|H\|<\epsilon\right\}$.
		      Consequently, $\partial X$ is closed at $x$; that is, if $\lim _{k \rightarrow+\infty} x_{k}=x, H_{k} \in \partial X\left(x_{k}\right)$ for all $k=0,1, \ldots$, and $\lim _{k \rightarrow+\infty} P_{x_{k} x} H_{k}P_{x x_k}=H$, then $H \in \partial X(x)$.
	\end{enumerate}
\end{proposition}

\subsection{Trust Region Methods}

Trust region methods are an extension of Newton's method, which have better convergence properties and relax the convexity requirement by automatically detecting the negative curvature. For a comprehensive discussion, we refer readers to monographs such as \cite{nocedalNumericalOptimization2006} and \cite{conn2000Trustregion}.
While trust region methods share the same objective function as Newton's method in the model problem, they possess a trust region constraint.
Specifically, for a smooth function $\phi$ in a Euclidean space $E$, a classical trust region method (see \cite{conn2000Trustregion}) often chooses the model problem at $x_k$ as
\begin{equation}\label{eq:etr}
	\begin{aligned}
		\min_{\eta\in E} & \quad m_{x_k}(\eta) \coloneqq \phi(x_k) + \langle\grad_{E}\phi(x_k), \eta\rangle_{E} + \frac{1}{2}\langle\operatorname{Hess}_{E}\phi(x_k)\eta,\eta\rangle_{E} \\
		\text{s.t.}      & \quad \|\eta\|_E\le \Delta_k,
	\end{aligned}
\end{equation}
where $\grad_{E}$, $\operatorname{Hess}_{E}$, $\left<\cdot,\cdot  \right>_{E}$, and $\|\cdot \|_{E}$ are the Euclidean gradient, Hessian, inner product, and norm respectively.
After solving the model problem, a trust region method compares the actual decrease and model decrease by computing the relative decrease ratio, which is used to determine the next iteration point and trust region radius. We defer the implementation details of a trust region method to the next section.
% \FloatBarrier

\section{Semismooth Riemannian Trust Region Method} \label{sec:alg}
In this section, we present a semismooth Riemannian trust region method to solve the unconstrained nonconvex problem on a manifold:
\begin{align*}
	\min_{x\in\M} \varphi(x),
\end{align*}
where $\varphi$ is bounded below on $\mathcal{M}$, has a Lipschitz continuous and locally directionally differentiable gradient field with a Lipschitz constant $L$, but may not be twice differentiable.
To overcome the absense of the Hessian, we utilize the Clarke generalized covariant derivative of the gradient field in our trust region method. To ensure the super-linear local convergence rate, we require $\phi$ to be SC$^{1}$, i.e., impose the semismoothness condition on the Clarke generalized covariant derivative $\pat\grad\phi$.

\begin{definition}[Semismoothness \cite{deoliveiraNewtonMethodFinding2020}]\label{def:ss}
	Let $X$ be a locally Lipschitz continuous vector field on $\M$ that is directionally differentiable in a neighborhood of $x \in \M$.
	$X$ is said to be $\mu$-order semismooth at $x$ with respect to its Clarke generalized covariant derivative, for $\mu \ge 0$, if for any $\epsilon > 0$, there exists $\delta > 0$ such that
	$$
		\left\|X(x) - P_{yx}\left(X(y) + H_y \exp_y^{-1}(x)\right)\right\| \le \epsilon \dist (x,y)^{1+\mu},
	$$
	for all $y\in B_{\delta}(x)$ and $H_{y}\in \partial X(y)$, \re{where $\exp_{y}^{-1}$ is the inverse map of $\exp_{y}$, which is well-defined at $x$ for any $y$ in the normal neighborhood of $x$}.
	When $\mu = 0$, we simply say $X$ is semismooth. Moreover, if there exists $C,\delta > 0$ such that
	$$
		\left\|X(x) - P_{yx}\left(X(y) + H_y \exp_y^{-1}(x)\right)\right\| \le C \dist (x,y)^{2},
	$$
	{ for all $y\in B_{\delta}(x)$ and $H_{y}\in \partial X(y)$,} we say $X$ is strongly semismooth.
\end{definition}
For example, a piecewise smooth vector field $X$ is strongly semismooth with respect to $\pat X$.

When solving manifold optimization problems iteratively, we typically do not compute the iteration points directly on the manifold. Instead, we use Riemannian gradients and Hessians to calculate iteration points on the tangent space and then retract them onto the manifold. The exponential map, a distance-preserving smooth diffeomorphism between the tangent space and the manifold, ensures that the retracted points preserve desirable properties, such as sufficient descent in the objective function.
However, computing the exponential map can be challenging.
To address this, we introduce retractions, a class of mappings that approximate the exponential map and relax the requirement for a distance-preserving smooth diffeomorphism.

\begin{definition}[Retraction] \label{def:r}
	A continuously differentiable mapping $R: T\M \to \M$ is called a retraction, if for any $x\in\M$, it satisfies that
	\begin{enumerate}
		\item $R_x(0_x) = x$,
		\item $\d R_x(0_x) = \operatorname{id}_{T_x\M}$,
	\end{enumerate}
	where $R_x$ is the restriction of $R$ to $\tm{x}$, $0_x$ is the zero element of $\tm{x}$, and $\d R_x(0_x)$ is the differential (pushforward) of $R_x$ at $0_x$.
\end{definition}

The definition of a retraction shows that it provides a first-order approximation of the exponential map, which also satisfies $\exp_{x}(0_{x}) = x$ and $\d\exp_{x}(0_{x}) = \operatorname{id}_{\tm{x}}$.
\re{In other words,} $R_x$ needs to map the point~$x + \xi$ in the tangent space back to the manifold in a way that preserves the distance between the two points as a higher-order term compared to the magnitude of the tangent vector~$\xi$.
This condition ensures that $x + \xi$ and $R_x(\xi)$ exhibit similar properties. For example, if a continuous objective function exhibits a sufficient decrease for $x + \xi$, it should also demonstrate an acceptable decrease for $R_x(\xi)$.

\iffalse
Despite this, the above definition may seem abstract. To aid in comprehension, we provide an equivalent definition of retractions for embedded manifolds in Euclidean spaces. While we do not limit our discussion to such manifolds, this equivalence can help clarify the concept of retractions.

\begin{proposition}[{\Citep[Definition 3.4]{chen2020proximal}}] \label{prop:r}
	When $\M$ is an embedded manifold in vector space $\R^{k}$ and its Riemannian inner product is induced by the dot product, then the second condition in \cref{def:r} is equivalent to
	\begin{equation*}
		\lim_{T_{x}\M \ni \xi \to 0_{x}} \frac{\| R_{x}(\xi) - (x + \xi)\|_{2}}{\|\xi\|_{2}} = 0.
	\end{equation*}
\end{proposition}

\cref{prop:r} helps translate the conditions on a retraction:

% Retractions are a useful tool for manifold optimization, as indicated by \cref{prop:r} and other desirable properties (see \cite[Chapter 4.1]{absilOptimizationAlgorithmsMatrix2008}). They turn the tangent space into a first-order approximation of the manifold, with the exponential map being a special retraction.
\fi

To achieve a quadratic local convergence rate, the retraction needs to be $C^{2}$ \cite{absilOptimizationAlgorithmsMatrix2008,absilTrustRegionMethodsRiemannian2007}.
However, if the gradient field of the objective function is not strongly semismooth, the quadratic local convergence rate may not be obtained.
Hence, to be more general and consistent with the semismoothness condition,
we only require the retraction to admit a \holder continuous differential, rather than being twice continuously differentiable. 

{%
	\begin{remark}[$C^{1,\nu}$ retractions]
		We present two examples of $C^{1,\nu}$ retractions, which are retractions with a $\nu$-order \holder continuous differential, on the real coordinate space $\R^{n}$.
		By introducing two $C^{2,\nu}$ Riemannian metrics, we naturally have their associated exponential maps as $C^{1,\nu}$ retractions.
	First, note that on $\R^{n}$, a metric (a family of inner products) can be represented by a mapping $g\colon \R^{n} \to S_{+}^{n}$, where $S_{+}^{n}$ is the set of positive definite matrices.
			The resulting inner product at $p$ is $\left< \xi,\eta \right>_{x} = \sum_{i,j}g_{ij}(p)\xi_{i}\eta_{j}$ for any $\xi, \eta \in T_{p}\mathcal{M}$, where $g_{ij}(p)$ is the $(i,j)$-th entry of $g(p)$.
				Then, a toy example of a $C^{2,\mu}$ metric on $\mathcal{M} = \R^{n}$ is given by $g(p) = f(p) \cdot I_{n}$, where $f\colon \R^{n} \to \R$ is $C^{2,\mu}$.
				An example of an $f\in C^{2,\mu}$ can be found in, e.g., \cite[Example 1]{grapiglia2017RegularizedNewton}.

				Another more general example is the induced metric on a hypersurface defined by a function graph $\mathcal{M} = \left\{ (x, f(x)) \mid x \in \R^{n-1} \right\}$, where $f\colon \R^{n-1} \to \R$.
				The Riemannian metric at $p = (x,f(x))\in\M$ induced by the Euclidean metric on $\R^{n}$ is $g_{ij}(p) = \delta_{ij} + \partial_{i}f(x)\partial_{j}f(x)$, where $\delta_{ij}$ is the Kronecker delta \citep{huang2013hypersurfaces,lee2019Geometricrelativity}.
				Therefore, if $f$ is $C^{3,\mu}$, then $g$ is $C^{2,\mu}$, and the exponential map associated with $g$ is $C^{1,\mu}$.
	\end{remark}
}%
We have the following proposition for this class of retractions.

\begin{proposition}\label{prop:r-ss}
	Suppose the differential of $R_{x}$ is $\nu$-order H\"older continuous,
	i.e., there exists $\nu >0$ and $C\ge 0$ such that for any $\xi_1,\xi_2\in\tm{x}$, we have
	$$
		\|P_{R_{x}(\xi_1)R_{x}(\xi_2)}\d R_{x}(\xi_1) - \d R_{x}(\xi_2)\|_{\operatorname{op}} \le C\|\xi_1 - \xi_2\|^{\nu}
		.$$
		Then for any $\xi\in\tm{x}$ \re{such that $R_{x}(\xi), \exp_{x}(\xi)\in \operatorname{NN}(x)$}, we have
	$$
		\dist (R_x(\xi), \exp_x(\xi)) = O \left(\|\xi\|^{1 + \nu}\right).
	$$
\end{proposition}
\begin{proof}
	For any $\xi$, define the curve $\gamma: t \mapsto R_{x}(t\xi)$. For any smooth function $f$ on the manifold, denote $\hat{f}(t) = f(R_{x}(t\xi))$.
	By the mean value theorem, we know there exists $\tau\in[0,1]$ such that
	\begin{equation}\label{eq:r-ss-1}
		\hat{f}(1) = \hat{f}(0) + \hat{f}'(\tau)
		\!= f(x) + \left<\grad f(y), \gamma'(\tau) \right>
		\!= f(x) + \left<\grad f(y), \d R_{x}(\tau\xi)[\xi]\right>,
	\end{equation}
	where $y = R_{x}(\tau\xi)$, and $\d R_{x}(\tau\xi): \tm{x} \to \tm{y}$.
	Then, we have
	$$
		\begin{aligned}
			  & \left<\grad f(y), \d R_{x}(\tau\xi)[\xi]\right>
			= \left<P_{yx}\grad f(y), P_{yx}\d R_{x}(\tau\xi)[\xi] \right>                                                                                                                                             \\
			= & \left<\grad f(x), \d R_{x}(0_{x})[\xi]\right>                                                       + \underbrace{\left<\grad f(x), P_{yx}\d R_{x}(\tau\xi)[\xi] - \d R_{x}(0_{x})[\xi] \right>}_{S_1} \\
			  & +  \underbrace{\left<P_{yx}\grad f(y) - \grad f(x), P_{yx}\d R_{x}(\tau\xi)[\xi] \right>}_{S_2}.
		\end{aligned}
	$$
	By the H\"older continuity of $\d R_{x}$, $S_1 = O(\|\xi\|^{1+\nu})$; and by the smoothness of $f$, $\|P_{yx}\grad f(y) - \grad f(x)\| = O(\dist(x,y)) = O(\|\xi\|)$, and then $S_2 = O(\|\xi\|^2)$. Then, using $\d R_{x}(0_{x}) = \operatorname{id}_{\tm{x}}$ in \cref{def:r}, we get
	\begin{equation}\label{eq:r-ss-2}
		\left<\grad f(y), \d R_{x}(\tau\xi)[\xi] \right> = \left<\grad f(x), \xi \right> + O(\|\xi\|^{1+\nu}).
	\end{equation}
	Now let $f(p) \coloneqq \dist(p,\exp_{x}(\xi))$.
	\re{
	$f$ is the \textit{radial} distance function on $\operatorname{NN}(\exp_{x}(\xi))$ and is differentiable on $\operatorname{NN}(\exp_{x}(\xi)) \setminus \left\{ \exp_{x}(\xi) \right\}$ \cite[Lemma 6.8]{lee2012Smoothmanifolds}.
	The geodesic connecting $R_{x}(\xi)$ and $\exp_{x}(\xi)$ is unique and minimizing the distance given that both points are in $\operatorname{NN}(x)$. Thus, $f$ is well-defined on $R_{x}(\xi)$.
}
	Combining \cref{eq:r-ss-1,eq:r-ss-2} gives
	\begin{equation}\label{eq:r-ss-3}
		\dist(R_{x}(\xi),\exp_{x}(\xi)) = \hat{f}(1) = \dist(x,\exp_{x}(\xi)) + \left<\grad f(x), \xi \right> + O(\|\xi\|^{1+\nu}).
	\end{equation}
	Then, \re{by \cite[Theorem 6.32]{lee2018IntroductionRiemannian}}, the gradient of the Riemannian distance function gives
  \begin{equation}\label{eq:r-ss-4}
    \left<\grad f(x), \xi\right> = \left<\gamma_-'(\dist(x,\exp_{x}(\xi))),\xi \right> 
  \end{equation}
	\re{where $\gamma_-$ is the unit-speed geodesic from $\exp _{x}(\xi)$ to $x$. Let $\gamma$ be the reverse curve of $\gamma_{-}$, i.e., $\gamma$ is the same curve as $\gamma_{-}$ but with a reverse parametrization from $x$ to $\exp_{x}(\xi)$.
	Then, we know that
	\begin{equation}\label{eq:r-ss-5}
	  \gamma_-'(\dist (x, \exp_{x}(\xi))) = -\gamma'(0) = -\frac{\xi}{\|\xi\|}
	,\end{equation}
	where the negative sign is due to the reverse parametrization of $\gamma$, and the normalization is because $\gamma$ and $\gamma_{-}$ are unit-speed geodesics.
	Combining \cref{eq:r-ss-4,eq:r-ss-5} gives
	\begin{equation}\label{eq:r-ss-6}
	  \left<\grad f(x), \xi\right> = -\|\xi\| = -\left<\frac{\xi}{\|\xi\|},\xi \right> = -\dist(x,\exp_{x}(\xi))
	.\end{equation}
}
  Combining \cref{eq:r-ss-3,eq:r-ss-6} gives
  $$
    \dist(R_{x}(\xi),\exp_{x}(\xi))
    = O(\|\xi\|^{1+\nu}).
  $$
\end{proof}

Using the retraction, we can first solve the model problem of a trust region method on the tangent space and then map the iteration point back onto the manifold.
The model problem of a Riemannian trust region method can be defined similarly to \eqref{eq:etr} using the Riemannian gradient, Hessian, inner product, and norm \cite{absilTrustRegionMethodsRiemannian2007}.
However, the objective function $\phi$ we consider may not necessarily be twice differentiable. As a result, we propose replacing the Hessian in the model problem with an arbitrary element in the Clarke generalized covariant derivative of the gradient vector field. That is, at each iteration, we choose an arbitrary $H_k \in \pat\grad\phi(x_k)$\footnotemark and define the model problem as follows:
\begin{equation}\label{eq:model}
	\begin{aligned}
		\min_{\eta\in\tm{x_k}} & \quad m_{x_k}(\eta) \coloneqq
		\phi(x_k) + \lang \grad \phi(x_k), \eta\rang + \frac{1}{2} \lang H_{k} \eta,\eta\rang \\
		\text{s.t.} \ \        & \ \  \|\eta\|\le \Delta_k.
	\end{aligned}
\end{equation}

\footnotetext{Note that we sometimes use $H_{x}$ to represent an element in $\pat\grad\phi(x)$ to make it more self-explanatory. Thus, $H_{x_k}$ and $H_k$ both represent an element in $\pat\grad\phi(x_k)$.}

After solving the model problem, we compare the descent in the objective function with that of the model function by computing the relative decrease ratio:
\begin{equation}\label{eq:rho}
	\rho_k = \frac{\phi(x_k) - \phi(R_{x_k}(\eta_k))}{m_{x_k}(0) - m_{x_k}(\eta_k)},
\end{equation}
where $\eta_k$ is the (approximate) solution to the model problem \eqref{eq:model}, and $R$ is the chosen retraction. If $\rho_k$ is relatively large, we accept $R_{x_k}(\eta_k)$ as the next iteration point. The remaining steps of our semismooth Riemannian trust region method are the same as in a vanilla trust region method \citep{nocedalNumericalOptimization2006}. We present our algorithm in Algorithm~\ref{alg:sstr}.

For the model problem, any approximate method with an appropriate termination condition can be used. For example, we use the truncated conjugate gradient (TCG) method \cite{steihaugConjugateGradientMethod1983,toint1981efficientsparsity} with the following stopping criterion
\begin{equation} \label{eq:stop}
	\|r_{j+1}\| \le \|r_0\|\min \{\|r_0\|^\theta, \epsilon\},
\end{equation}
where $\epsilon,\theta > 0$.
For completeness and convenience of subsequent analysis, we present the TCG method in Algorithm~\ref{alg:tcg}.
Besides stopping criterion \eqref{eq:stop}, Algorithm~\ref{alg:tcg} also terminates when one of two truncation conditions (lines 3 and 8) is satisfied and returns the truncated tangent vector $\eta_k = \xi_{j} + \tau\delta_{j}$, with $\tau$ calculated as follows:
\begin{equation} \label{eq:tau}
	\tau(\xi_j, \del_j, \Delta_k) = \frac{-\lang\xi_j, \delta_{j}\rang+\sqrt{\lang\xi_j, \delta_{j}\rang^{2}+\left(\Delta_k^{2}-\lang\xi_j, \xi_j\rang\right) \lang\delta_{j}, \delta_{j}\rang}}{\lang\delta_{j}, \delta_{j}\rang}.
\end{equation}

\begin{algorithm}[!th]
	\caption{Semismooth Riemannian Trust Region Method}\label{alg:sstr}
	\SetKw{Para}{parameters}
	\SetKw{Input}{input}
	\Para{$\bar{\Delta} > 0, \Delta_0 \in (0,\bar{\Delta})$, and $\rho'\in[0,1/4)$}\;
	\Input{initial point $x_0\in\mathcal{M}$}\;
	\For{$k=0,1,\dots$}{
		Obtain $\eta_k$ by approximately solving the model problem \eqref{eq:model}, e.g., using Algorithm~\ref{alg:tcg}\;
		Compute $\rho_k$ using \eqref{eq:rho}\;
		\uIf {$\rho_k < 1/4$} {
			$\Delta_{k+1} = \frac{1}{4}\Delta_{k}$
		} \uElseIf {$\rho_k > 3/4$ \textup{and} $\|\eta_k\| = \Delta_k$} {
			$\Delta_{k+1} = \min\{2\Delta_k, \bar{\Delta}\}$
		} \Else {
			$\Delta_{k+1} = \Delta_k$
		}
		\uIf {$\rho_k > \rho'$} {
			$x_{k+1} = R_{x_k}(\eta_k)$
		} \Else {
			$x_{k+1} = x_k$
		}
	}
\end{algorithm}

\begin{algorithm}[!th]
	\caption{Steihaug-Toint Truncated Conjugate Gradient Method}\label{alg:tcg}
	Let $\xi_{0}= 0, r_0 = \grad \phi(x_k), \delta_0 = -r_0$\;
	\For {$j= 0,1,2,\dots$} {
	\If {$\lang \delta_j , H_k \delta_j \rang \le 0$} {
		\Return $\eta_k = \xi_{j}+ \tau\delta_j$, where $\tau$ is computed using \eqref{eq:tau}\;
	}
	Let  $\al_j = \lang r_j, r_j\rang / \lang \del_j, H_k \del _j\rang$\;
	Let $\xi_{j+1} = \xi_{j}+ \al_j\del_j$\;

	\If {$\|\xi_{j+1}\|\ge \Delta_k$ } {
		\Return $\eta_k = \xi_{j}+ \tau\delta_j$, where $\tau$ is computed using \eqref{eq:tau}\;
	}

	Let $r_{j+1} = r_j + \al_j H_k \del_j$\;

	\If {stopping criterion \eqref{eq:stop} is met} {
		\Return $\eta_k = \xi_{j+1}$
	}
	Let $\beta_{j+1} = \lang r_{j+1}, r_{j+1} \rang / \lang r_j, r_j \rang$\;
	Let $\del_{j+1} = -r_{j+1} + \beta_{j+1} \del_j$\;
	}
\end{algorithm}

\section{Convergence Analysis} \label{sec:anlys}
In this section, we present the convergence results of our semismooth Riemannian trust region method (Algorithm \ref{alg:sstr}).
We prove three classical results that are applicable to a smooth Euclidean trust region method.
The first is the global convergence theorem (\cref{thm:global}), which shows that the algorithm converges to a stationary point for any initial point.
The second is the local convergence theorem (\cref{thm:local}), which demonstrates that nondegenerate local minimizers form basins of attraction.
Finally, \cref{thm:rate} establishes the super-linear local convergence rate of our algorithm.

\subsection{Global Convergence}

Before presenting the global convergence theorem, we need some essential lemmas.
Since both the model problem~\cref{eq:model} and Algorithm~\ref{alg:tcg} are defined in a Euclidean space, the first two lemmas apply to our algorithm without requiring any modification.

\begin{lemma}[TCG properties {\citep[Theorem 2.1]{steihaugConjugateGradientMethod1983}}]
	\label{lem:tcg}
	Let $\{\xi_i\}_{i=0}^{j}$
	be the first $j+1$ tangent vectors generated by Algorithm~\ref{alg:tcg} with $j+1$ iterations and $\eta_k$ be the returned tangent vector. Then we have
	\begin{enumerate}
		\item If the truncation conditions and the termination condition are not met, there exists $j$ such that
		      $$
			      \xi_{j+1} = \eta^* = H_k^{-1}(-\grad \phi(x_k)),
		      $$
		      where $H_k$ is the (general/approximated) Hessian passed to the algorithm, and $\eta^*$ is the minimum point of $m$.
		\item Algorithm~\ref{alg:tcg} is a descent algorithm, i.e., $m(\xi_0) \ge m(\xi_1) \ge \dots \ge m(\xi_i) \ge \dots \ge m(\xi_j) \ge m(\eta_k) \ge m(\eta^{*})$.
		\item The norm of $\xi_i$ is monotonically increasing, i.e., $\|\xi_0\| \le \|\xi_1\| \le \dots \le \|\xi_i\| \le \dots\le \|\xi_{j}\| \le \|\eta_k\| \le \|\eta^{*}\|$.
	\end{enumerate}
\end{lemma}

\begin{lemma}[Cauchy decrease inequality {\citep[Lemma 4.3]{nocedalNumericalOptimization2006}}] \label{lem:cauchy}
	Let $\eta_k$ be the tangent vector returned by Algorithm~\ref{alg:tcg}, then the decrease in the model problem~\cref{eq:model} satisfies
	\begin{equation*}
		m_{x_k}(0) - m_{x_k}(\eta_k) \ge \frac{1}{2}\|\grad \phi(x_k)\| \min\left\{\Delta_k, \frac{\|\grad \phi(x_k)\|}{\|H_{x_k}\|}\right\}.
	\end{equation*}
\end{lemma}

The following lemma, Taylor's theorem, is of utmost importance in our analysis. While there exist several forms and variations of Taylor's theorem on manifolds, we will only present the ones that are relevant to our analysis. Some other variations can be found in \cite{absilOptimizationAlgorithmsMatrix2008,boumalIntroductionOptimizationSmooth2020a,hiriart-urruty1984GeneralizedHessian}.

\begin{lemma}[Taylor]\label{lem:taylor}
	Suppose $\phi\in C^1(\mathcal{M})$ has a Lipschitz gradient field and $x, y \in \mathcal{M}$. Let $\gamma: [0, 1] \to~\mathcal{M}$ be a geodesic from $x$ to $y$. Then, there exist $\tau_1, \tau_2\in [0, 1]$, and $H_{\tau_2} \in\pat\grad{\phi(\gamma(\tau_2))}$ such that
	\begin{equation}\label{eq:taylor-1}
		\phi(y) - \phi(x) = \lang P^{\tau_1\to0}_\gamma \grad \phi(\gamma(\tau_1)), \gamma'(0)\rang,
	\end{equation}
	\begin{equation}\label{eq:taylor-2}
		\phi(y) - \phi(x) = \lang \grad \phi(x), \gamma'(0)\rang + \frac{1}{2} \lang P^{\tau_2\to 0}_\gamma H_{\tau_2} P^{0\to \tau_2}_\gamma \gamma'(0), \gamma'(0)\rang,
	\end{equation}

	Furthermore, we have
	\begin{equation}
		\label{eq:taylor-3}
		\phi(y) - \phi(x) = \lang\grad \phi(x), \gamma'(0)\rang + O(\dist(x,y)^2),
	\end{equation}
	\begin{equation}
		\label{eq:taylor-4}
		\phi(y) - \phi(x) = \lang \grad \phi(x), \gamma'(0)\rang + \frac{1}{2} \lang H_x \gamma'(0), \gamma'(0)\rang + o(\dist(x,y)^2),
	\end{equation}
	for some $H_x \in \pat\grad{\phi(x)}$.

	Similarly, for Lipschitz and directionally differentiable vector field $X$, there exist $\tau_3 \in [0,1]$, $H_{\tau_3} \in \pat\grad{\phi(\gamma(\tau_3))}$, and $H_x\in\pat\grad{\phi(x)}$ such that
	\begin{equation} \label{eq:taylor-5}
		X(y) = P_\gamma^{0\to 1}\left[X(x) + P_\gamma^{\tau_3\to 0} H_{\tau_3}P_\gamma^{0\to \tau_3}\gamma'(0)\right],
	\end{equation}
	\begin{equation} \label{eq:taylor-6}
		\left\|X(y) - P_\gamma^{0\to 1}\left[X(x) + H_x\gamma'(0)\right] \right\| = o(\dist(x,y)).
	\end{equation}
\end{lemma}

\begin{proof}
	For \cref{eq:taylor-1}, let $\gamma$ be the geodesic from $x$ to $y$ and define $\hat{\phi} \coloneqq \phi\circ\gamma$. Then by the first-order expansion of $\hat{\phi}$, there exists $\tau_1\in[0,1]$ such that
	$$
		\begin{aligned}
			\phi(y) = \hat{\phi}(1) = \hat{\phi}(0) + \hat{\phi}'(\tau_1)
			= & \phi(x) + \left<\grad \phi(\gamma(\tau_1)),\gamma'(\tau_1) \right>                     \\
			= & \phi(x) + \left<P^{\tau_1\to 0}_{\gamma}\grad \phi(\gamma(\tau_1)),\gamma'(0) \right>,
		\end{aligned}
	$$
	where the last equality is from the fact that the tangent vector of a geodesic is parallel along itself (see \cref{def:geo-para}).

	For the proof of \eqref{eq:taylor-2}, please refer to \cite[Lemma 4.1]{zhou2022semismooth}. Let $L$ be the Lipschitz constant of $\phi$'s gradient as in \cref{def:lip}; by \eqref{eq:taylor-1}, we have
	$$\begin{aligned}
			\left\|\phi(y) - \phi(x) - \lang \grad \phi(x), \gamma'(0)\rang \right\|
			 & = \left\|\lang \left(P_{\gamma}^{\tau_{1}\to 0}\grad \phi(\gamma(\tau_1)) - \grad \phi(x) \right), \gamma' (0)\rang\right\| \\
			 & \le \left\|P_{\gamma}^{\tau_{1}\to 0}\grad \phi(\gamma(\tau_1)) - \grad \phi(x) \right\|\|\gamma'(0)\|                      \\
			 & \le L \dist(\gamma(\tau_1), x) \|\gamma'(0)\|                                                                               \\
			 & \le L \dist(x,y)^2,
		\end{aligned}
	$$
	which gives \eqref{eq:taylor-3}.
	By the upper-semicontinuity of the Clarke generalized covariant derivative (Proposition \ref{lem:cgcd}), for any $\epsilon > 0$, when $y$ is near $x$, there exist $H_{x}\in \pat\grad\phi(x)$ and an operator $B$ whose operator norm is no greater than 1, such that
	$$
		P^{\tau_2\to 0}_\gamma H_{\tau_2} P^{0\to \tau_2}_\gamma = H_x + \epsilon B.
	$$
	Therefore
	$$\begin{aligned}
			                                  & \left\|\phi(y) - \phi(x) - \lang \grad \phi(x), \gamma'(0)\rang - \frac{1}{2} \lang H_x\gamma'(0), \gamma'(0)\rang\right\|            \\
			\overset{\cref{eq:taylor-2}}{\le} & \frac{1}{2}\left\|\lang\left(P^{\tau_2\to 0}_\gamma H_{\tau_2} P^{0\to \tau_2}_\gamma - H_x\right)\gamma'(0), \gamma'(0)\rang\right\| \\
			\le                               & \frac{1}{2}\|P^{\tau_2\to 0}_\gamma H_{\tau_2} P^{0\to \tau_2}_\gamma - H_x\|\|\gamma'(0)\|^2                                         \\
			\le                               & \frac{\epsilon}{2}\dist(x,y)^2.
		\end{aligned}$$
	By the arbitrariness of $\epsilon$, we get \eqref{eq:taylor-4}.
	Equations \eqref{eq:taylor-5} and \eqref{eq:taylor-6} can be derived similarly.
\end{proof}

We now present the global convergence theorem, which mirrors \citep[Theorem~4.4]{absilTrustRegionMethodsRiemannian2007}. However, in \cite{absilTrustRegionMethodsRiemannian2007}, two additional assumptions are made on the retraction $R$. We could also adopt the same assumptions and then \citep[Theorem~4.4]{absilTrustRegionMethodsRiemannian2007} directly applies to our setting, as it does not require the second-order differentiability of $\phi$.
Nonetheless, we provide a proof of our theorem with only one assumption: the H\"older continuity of the retraction's differential. This requirement is strictly weaker than the radially Lipschitz continuity assumption made in \citep[Definition~4.1]{absilTrustRegionMethodsRiemannian2007}.

\begin{theorem}[Global convergence]\label{thm:global}
	Let $\{x_k\}$ be the sequence generated by Algorithm~\ref{alg:sstr} with $\rho'\in [0,\frac{1}{4})$. Suppose on the level set $\{x\in \mathcal{M}: \phi(x) \le \phi(x_0) \}$, $\{H_{k}\}$ is uniformly bounded. Then we have
	\begin{equation}\label{eq:glob-1}
		\liminf_{k\to\infty} \|\grad \phi(x_k)\| = 0.
	\end{equation}
	Moreover, if $\rho'\in(0,\frac{1}{4})$ and the retraction $R$ has a $\nu$-\holder continuous differential, we have
	\begin{equation}\label{eq:glob-2}
		\lim_{k\to\infty} \|\grad \phi(x_k)\| = 0.
	\end{equation}
\end{theorem}

\begin{proof}
	Our proof is adapted from that of \citep[Theorem~4.2 and 4.4]{absil2006convergence}.
	For \cref{eq:glob-1}, we only need to reestablish the claim that the trust region radius has a positive lower bound if $\liminf_{k\to\infty}\|\grad \phi(x_k)\| \ne 0$. For the remaining proof, please refer to \citep[Theorem~4.2]{absil2006convergence}.
	Similar to the proof of \eqref{eq:taylor-1} in \cref{lem:taylor}, let $\gamma$ be the curve such that $\gamma(t) = R_{x_k}(t\eta_k)$. Then, there exists $\tau\in[0,1]$ such that
	\begin{align}\notag
		\phi(R_{x_k}(\eta_k)) - \phi(x_k)
		= & \left<\grad\phi(\gamma(\tau)), \gamma'(\tau) \right>                                                         \\\notag
		= & \left<\grad\phi(\gamma(\tau)), \d R_{x_k}(\tau\eta_k)[\eta_k] \right>                                        \\\label{eq:glob-3}
		= & \left<P_{\gamma(\tau)x_k}\grad\phi(\gamma(\tau)), P_{\gamma(\tau)x_k}\d R_{x_k}(\tau\eta_k)[\eta_k] \right>.
	\end{align}
	By the \holder continuity of $\d R_{x_k}$, there exist $C_1 \ge 0$ and $\nu>0$ such that
	\begin{equation}\label{eq:glob-4}
		\|\d R_{x_k}(0_{x_k}) - P_{\gamma(\tau)x_k}\d R_{x_k}(\tau\eta_k)\|_{\mathrm{op}} \le C_1\|\tau\eta_k\|^{\nu}.
	\end{equation}
	By the Lipschitz continuity of $\grad\phi$ and \cref{prop:r-ss}, there exists $C_2 \ge 0$ such that
	\begin{align}\notag
		    & \|\grad\phi(x_k) - P_{\gamma(\tau)x_k}\grad\phi(\gamma(\tau))\|                            \\\notag
		\le & L\dist(x_k, R_{x_k}(\tau\eta_k))                                                           \\\notag
		\le & L\dist(x_k, \exp _{x_k}(\tau\eta_k)) + L\dist(\exp_{x_k}(\tau\eta_k), R_{x_k}(\tau\eta_k)) \\\label{eq:glob-5}
		\le & L(\tau\|\eta_k\| + C_2\|\tau\eta_k\|^{1+\nu})
	\end{align}
	Plugging {\cref{eq:glob-4,eq:glob-5}} back into \eqref{eq:glob-3} gives
	$$
		\begin{aligned}
			\quad & \phi(R_{x_k}(\eta_k)) - \phi(x_k)                                                                                          \\
			=     & \left<P_{\gamma(\tau)x_k}\grad\phi(\gamma(\tau)) - \grad\phi(x_k),
			P_{\gamma(\tau)x_k}\d R_{x_k}(\tau\eta_k)[\eta_k]\right>                                                                           \\
			      & + \left<\grad\phi(x_k), (P_{\gamma(\tau)x_k}\d R_{x_k}(\tau\eta_k) - \d R_{x_k}(0_{x_k}))[\eta_k] \right>
			+ \left<\grad\phi(x_k), \eta_k\right>                                                                                              \\
			\le   & L(\tau\|\eta_k\| \!+\! C_2\|\tau\eta_k\|^{1\!+\!\nu})\|P_{\gamma(\tau)x_k}\d R_{x_k}(\tau\eta_k)\|_{\mathrm{op}}\|\eta_k\| \\
			      & +
			\|\grad\phi(x_k)\|\!\cdot\! C_1\|\tau\eta_k\|^{\nu}\!\cdot\! \|\eta_k\| +
			\left<\grad\phi(x_k), \eta_k\right>                                                                                                \\
			\le   & L(\|\eta_k\| + C_2\|\eta_k\|^{1+\nu})(1 + C_1\|\tau\eta_k\|^{\nu})\cdot \|\eta_k\|                                         \\
			      & + C_1\|\grad\phi(x_k)\|\|\eta_k\|^{1+\nu} +
			\left<\grad\phi(x_k), \eta_k\right>                                                                                                \\
			\le   & (C_3 + C_1\|\grad\phi(x_k)\|)\|\eta_k\|^{1+\nu} +
			\left<\grad\phi(x_k), \eta_k\right>,
		\end{aligned}
	$$
	where $C_3 \coloneqq L(\bar{\Delta}^{1-\nu} + C_2\bar{\Delta})(1 + C_1\bar{\Delta}^{\nu}) \ge L(\|\eta_k\|^{1-\nu} + C_2\|\eta_k\|)(1 + C_1\|\tau\eta_k\|^{\nu})$ because $\|\eta_k\|\le \bar{\Delta}$ and $\tau\in[0,1]$, where $\bar{\Delta}$ is the radius cap of the trust region specified in Algorithm~\ref{alg:sstr}.
	Let $\beta$ be the uniform upper bound of $\{ H_k \}$. By the definition of the model problem \eqref{eq:model}, we have
	$$\begin{aligned}
			\left|m_{x_k}(\eta_k) - \phi(R_{x_k}(\eta_k))\right|
			 & \le \left|\frac{1}{2}\lang \eta_k, H_k\eta_k\rang \right| +
			\left| \phi(x_k) + \left<\grad\phi(x_k),\eta_k \right> - \phi(R_{x_k}(\eta_k)) \right| \\
			 & \le \frac{\beta}{2}\|\eta_k\|^2 + (C_3 + C_1\|\grad\phi(x_k)\|)\|\eta_k\|^{1+\nu}   \\
			 & \le (C_1\|\grad\phi(x_k)\| + C_4)\|\eta_k\|^{1+\nu},
		\end{aligned}$$
	where $C_4 = C_3 + \beta\bar{\Delta}^{1-\nu}/2$. Then, by \cref{lem:cauchy}, we get
	$$
		\begin{aligned}
			|\rho_k - 1| = \left|\frac{m_{x_k}(\eta_k) - \phi(R_{x_k}(\eta_k))}{m_{x_k}(0) - m_{x_k}(\eta_k)}\right|
			\le \frac{(C_1\|\grad\phi(x_k)\| + C_4)\|\eta_k\|^{1+\nu}}{\frac{1}{2}\|\grad{\phi(x_k)}\|\min\left\{\Delta_k,\frac{\|\grad{\phi(x_k)\|}}{\|H_{k}\|}\right\}}.
		\end{aligned}
	$$
	Suppose the $\liminf_{k\to \infty}\|\grad\phi(x_k)\|\ne 0$, then there exist $\epsilon> 0$ and $K\in\mathbb{N}$ such that $\|\grad\phi(x_k)\|\ge \epsilon$ for all $k\ge K$. Then, for any $k\ge K$, we have
	$$
		|\rho_k - 1| \le \frac{C_5\|\eta_k\|^{1+\nu}}{\min\left\{\Delta_k, \epsilon /\beta \right\}},
	$$
	where $C_5 = 2(C_1 + C_4 /\epsilon)$.
	Let $\widetilde{\Delta} = \min\{\epsilon/\beta, (2C_{5})^{-1 /\nu}\}$. When $\Delta_k \le \widetilde{\Delta}$, we have $\min\{\Delta_k, \epsilon /\beta\} = \Delta_k$, and
	$$
		|\rho_k - 1| \le \frac{C_5 \Delta_k^{1+\nu}}{\Delta_k} \le C_5 \widetilde{\Delta}^{\nu} \le \frac{1}{2},
	$$
	which indicates that $\rho_k \ge 1/2 > \rho'$. Therefore, by the trust region radius update rule, we can conclude
	$$
		\Delta_{k+1}\begin{cases}
			\ge \Delta_k,                       & \text{ if } \Delta_k \le \widetilde{\Delta}; \\[2ex]
			\ge \frac{1}{4} \widetilde{\Delta}, & \text{ if } \Delta_k > \widetilde{\Delta}.
		\end{cases}
	$$

	That is, we establish the positive lower bound for the trust region radius: $\min\{\Delta_K, \widetilde{\Delta}/4\}$.

	For \cref{eq:glob-2}, we only need to reestablish the claim that there exist positive constants $C$ and $\delta$ such that for all $x\in\M$ and $\xi\in\tm{x}$ with $\|\xi\|\le \delta$, the following inequality holds:
	$$
		\|\xi\| \ge C\dist(x,R_{x}(\xi)).
	$$
	This claim is a direct consequence of \cref{prop:r-ss}. Specifically, there exist $C_1\ge0$ and $\nu>0$ such that
	$$
		\begin{aligned}
			\dist(x,R_{x}(\xi)) \le & \dist(x,\exp_{x}(\xi)) + \dist(\exp _{x}(\xi),R_{x}(\xi)) \\
			\le                     & \|\xi\| + C_1\|\xi\|^{1+\nu}.
		\end{aligned}
	$$
	Choose $\delta$ small enough such that $C_{1}\delta^{\nu} \le 1$, and we obtain
	$$
		\dist(x,R_{x}(\xi)) \le 2\|\xi\|,
	$$
	which implies that $\|\xi\| \ge 1 /2\dist(x,R_{x}(\xi))$ for all $x\in\M$ and $\xi\in\tm{x}$ with $\|\xi\|\le \delta$.
	Please refer to \citep[Theorem~4.4]{absil2006convergence} for other parts of the proof.
\end{proof}

In the statement of the theorem, we require that $\{ H_k \}$ is uniformly bounded. This seemingly strong condition holds under a mild assumption; we state it as a corollary.

\begin{corollary}
If there exists $k\in \mathbb{N}$ such that the level set $\{ x\in\M:\phi(x) \le \phi(x_k) \}$ is compact, then $\{ H_k \}$ is uniformly bounded due to the local boundedness of the Clarke generalized covariant derivative (see Proposition \ref{lem:cgcd}). Consequently, the condition required by Theorem \ref{thm:global} is met.
\end{corollary}

\begin{remark}
	\cref{thm:global} does not require the semismoothness of the gradient field of the objective function. Here, we only utilize the Lipschitz continuity of $\phi$'s gradient field.
	One can also prove the global convergence under a weaker assumption that the objective function is $C^{1}$ and satisfies a modified Kurdyka-Łojasiewicz condition on the manifold, for example, using techniques proposed in \cite{noll2013Convergencelinesearch}.
\end{remark}

\subsection{Local Convergence} \label{sec:local}

Theorem \ref{thm:global} states that the algorithm converges to some stationary point, which may not necessarily be a local minimizer.
Our next goal is to demonstrate that if the algorithm operates near a nondegenerate local minimizer, it will be \textit{attracted} to that local minimizer. To this end, we first introduce the definition of a nondegenerate local minimizer and then present some results highlighting how they shape the landscape of their neighborhoods.

\begin{definition}[Nondegenerate local minimizer]
	We say $x^*$ is a nondegenerate local minimizer of $\phi$, if $\grad \phi(x^*) = 0$ and $H_x$ is postive definite for any $H_x \in \pat\grad \phi(x^*)$.
\end{definition}

In \cref{thm:global}, we assume a uniform upper norm bound for $\{ H_k \}$. In the next lemma, we will demonstrate that near a nondegenerate local minimizer, $\pat\grad\phi$ automatically has some uniform bounds, not only on its norm, but also on its eigenvalues and the norm of its inverse.

\begin{lemma}[Local uniform bounds of the Clarke generalized covariant derivative]\label{lem:cgcd-bound}
	For a nondegenerate local minimizer $x^*$ of $\phi$, there exist $\lambda_1, \lambda_2 > 0$ and $\del > 0$ such that
	$$\begin{aligned}
			\lambda_1       & \le \min\left\{\left<\xi,H\xi\right>: H \in \partial\grad{\phi(x)}, x \in B_\del(x^*), \|\xi\|=1\right\}, \\
			\lambda_2       & \ge \max\left\{ \|H\|: H \in \partial\grad{\phi(x)}, x \in B_\del(x^*)\right\},                           \\
			\lambda_1 ^{-1} & \ge \max\left\{\|H^{-1}\|: H \in \partial\grad{\phi(x)}, x \in B_\del(x^*)\right\}.
		\end{aligned}$$
\end{lemma}

\begin{proof}
	For any $x\in\M$ and any postive definite $H\in\pat\grad\phi(x)$, $H$ and $H^{-1}$ are self-adjoint (see \cref{def:cgcd}). Thus, we have
	$$
		\begin{aligned}
			\min_{\|\xi\|=1} \left<\xi,H\xi \right> = & \lambda_{\min}(H) = (\lambda_{\max}(H^{-1}))^{-1} = \|H^{-1}\|^{-1}, \\
			\max_{\|\xi\|=1} \left<\xi,H\xi \right> = & \lambda_{\max}(H) = \|H\|,
		\end{aligned}
	$$
	where $\lambda_{\min}(H)$ and $\lambda_{\max}(H)$ are the smallest and the largest eigenvalues of $H$, respectively.
	By item~1 in \cref{lem:cgcd}, $\pat\grad\phi(x^*)$ is compact. Then, since all elements in $\pat\grad{\phi(x^*)}$ are positive definite and $\lambda_{\min}$ and $\lambda_{\max}$ are continuous functions on $\mathcal{L}(\tm{x^*})$, there exist $\lambda_1^*, \lambda_{2}^{*} > 0$ such that
	$$
		\begin{aligned}
			\lambda_{1}^{*} & = \min\{\lam_{\min}(H): H \in\pat\grad{\phi(x^*)}\}, \\
			\lambda_{2}^{*} & = \max\{\lam_{\max}(H): H \in\pat\grad{\phi(x^*)}\}.
		\end{aligned}
	$$
	By item 3 in \cref{lem:cgcd}, $\pat\grad\phi(x^*)$ is upper-semicontinuous. Let $\epsilon = \lambda_1^* /2$. Then there exists $\del_1 > 0$ such that for any $H \in \pat\grad{\phi(x)}$ and $x\in B_{\del_1}(x^*)$, there exist $\hat{H} \in \pat\grad{\phi(x^*)}$ and an operator $B$ with norm no greater than 1 such that
	$$
		P_{xx^*}HP_{x^*x} = \hat{H} + \frac{\lam_{1}^*}{2}B.
	$$
	Thus, for any $\xi \in \tm{x^*}\setminus\{ 0 \}$, we have
	$$
		\left<\xi, P_{xx^*}HP_{x^*x}\xi\right> = \left<\xi, \hat{H}\xi\right> + \frac{\lam_1^*}{2}\left<\xi, B\xi\right> \ge \lam_1^*\lang \xi,\xi\rang - \frac{\lam_1^*}{2}\lang \xi,\xi\rang = \frac{\lam_1^*}{2}\lang \xi,\xi\rang.
	$$
	Therefore $\lam_{\min}(H) = \lam_{\min}(P_{xx^*}HP_{x^*x}) \ge \lam_1^*/2$. By the arbitrariness of $H$ and $x$, we get
	$$
		\min\left\{\lam_{\min}(H) : H \in \partial\grad{\phi(x)}, x \in B_{\del_1}(x^*)\right\} \ge \lambda_1 \coloneqq \frac{\lambda_1^*}{2}.
	$$
	Similarly, we can obtain
	$$
		\max\left\{\lam_{\max}(H) : H \in \partial\grad{\phi(x)}, x \in B_{\del_1}(x^*)\right\} \le \lambda_{2} \coloneqq \lambda^{*}_{2} + \frac{\lambda_1^*}{2}.
	$$

	From the above bounds, we know that elements of $\pat\grad\phi(x)$ near $x^*$ are positive definite, and we have
	$$
		\max\{ \|H^{-1}\|: H\in\pat\grad\phi(x), x\in B_{\delta}(x^{*}) \} \le \lambda_{1}^{-1}.
	$$
\end{proof}

\begin{corollary} \label{lem:iso-min}
A nondegenerate local minimizer is an isolated local minimizer.
\end{corollary}

\begin{proof}
	Let $x^*$ be a nondegenerate local minimizer. By \eqref{eq:taylor-5} in \cref{lem:taylor} and \cref{lem:cgcd-bound}, there exists a neighborhood $U$ of $x^*$ such that for any $x \in U$,
	$$
		\grad{\phi(x)} = P_{\gamma}^{0\to 1}[\grad \phi(x^*) + P_\gamma^{\tau\to 0}H_{\tau}P_{\gamma}^{0\to\tau}\gamma'(0)] = P_{\gamma}^{\tau\to 1} H_{\tau}P_{\gamma}^{0\to \tau}\gamma'(0) \ne 0,
	$$
	where $\gamma$ is a geodesic joining $x^*$ and $x$, \re{$\tau\in[0,1]$, and $H_{\tau}\in\pat\grad\phi(\gamma(\tau))$}.
	Therefore, $x^*$ is the only stationary point in $U$.
\end{proof}

\begin{corollary} \label{lem:grad-bound}
Let $x^*$ be a nondegenerate local minimizer of $\phi$. Let $\lambda_1,\lambda_2$, and neighborhood $U$ be specified in \cref{lem:cgcd-bound}. We have the following relationship for any $x\in U$:
$$
	\lambda_1 \dist(x,x^*) \le \|\grad{\phi(x)}\| \le \lambda_2 \dist(x,x^*).
$$
\end{corollary}

\begin{proof}
	By \eqref{eq:taylor-5} in \cref{lem:taylor}, there exists $\tau\in[0,1]$ such that
	$$
		\|\grad{\phi(x)}\|\!=\!\|\grad{\phi(x^*)} + P_\gamma^{\tau\to0}H_{\gamma(\tau)}P_{\gamma}^{0\to\tau}\!\gamma'(0)\| \!=\!\|P_\gamma^{\tau\to0}H_{\gamma(\tau)}P_{\gamma}^{0\to\tau}\!\gamma'(0)\|,
	$$
	where $\gamma$ is the geodesic connecting $x,x^{*}$.
	Then, by \cref{lem:cgcd-bound}, we get
	$$
		\|\grad\phi(x)\| \begin{cases}
			\ge \min_{\|\xi\|=1} \|H_{\gamma(\tau)}\xi\|\|\gamma'(0)\|
			\ge \lambda_1\dist(x,x^*), \\
			\le \max_{\|\xi\|=1}\|H_{\gamma(\tau)}\xi\|\|\gamma'(0)\| \le \lambda_2\dist(x,x^*).
		\end{cases}
	$$
\end{proof}

We will now present and prove the local convergence theorem. Our proof is adapted from \citep[Theorem 4.12]{absilTrustRegionMethodsRiemannian2007}.

\begin{theorem}[Local convergence] \label{thm:local}
	Let $x^*$ be a nondegenerate local minimizer of $\phi$. Suppose the retraction $R$ has a $\nu$-\holder continuous differential near $x^*$. Then there exists a neighborhood $U$ of $x^*$ such that for any $x_0 \in U$, $\{x_k\}$ generated by Algorithm~\ref{alg:sstr} converges to $x^*$.
\end{theorem}

\begin{proof}
	By \cref{lem:iso-min}, there exists $\delta_0 > 0$ such that $x^*$ is the only local minimizer in $B_{\delta_0}(x^*)$.
	By \cref{lem:cgcd-bound}, there exist $c_1, \del_1 > 0$ such that $\|H^{-1}\| \le c_1$ for any $H \in \pat\grad \phi(x)$ and $x\in B_{\del_1}(x^*)$.
	Moreover, by \cref{lem:grad-bound}, there exist $c_2, \del_2 > 0$ such that $\|\grad \phi(x)\| \le c_2 \dist(x,x^*)$ for any $x\in B_{\del_2}(x^*)$.
	Since $B_{\delta_0}(x^{*})$ is compact,
	by \cref{prop:r-ss}, there exist $c_3, \delta_3 > 0$ such that for any $x\in B_{\delta_0}(x^{*})$ and $\xi \in B_{\delta_3}(0_{x})$, we have
	\begin{align}\notag
		\dist\left(R_{x}(\xi), x\right)
		\le & \dist(R_{x}(\xi),\exp_{x}(\xi)) + \dist(\exp_{x}(\xi),x) \\
		\label{eq:local-1}
		\le & O(\|\xi\|^{1+\nu}) + \|\xi\|
		\leq c_3\|\xi\|.
	\end{align}
	Let $\del_4 = \min\{\del_0,\del_1,\del_2,\frac{\delta_3}{c_1 c_2}\}$. Then let
	\begin{equation} \label{eq:local-2}
		\del = \frac{\del_4}{c_{1}c_2 c_3 + 1}.
	\end{equation}
	Since $\phi$ is continuous and $x^*$ is an isolated local minimizer by \cref{lem:iso-min}, there exists a level set $L = \{x\in \mathcal{M}: \phi(x) \le \phi(x^*) + \epsilon\}$ such that $L \cap B_{\del_4}(x^{*}) \subset B_\del(x^{*})$.
	Let $U = L \cap B_{\del_4}(x^{*})$. For any $x_0\in U$, let $\eta^* = H_0^{-1}\grad \phi(x_0)$. Then, we have
	\begin{equation} \label{eq:local-3}
		\|\eta^*\| \le \|H_0^{-1}\|\|\grad \phi(x_0)\| \le c_1c_2 \dist(x_0,x^*) \le c_1 c_2 \del \le c_1 c_2 \del_4 \le \delta_3.
	\end{equation}
	By \cref{lem:tcg}, $\|\eta_0\| \le \|\eta^*\| \le \delta_3$. Therefore, we have
	$$
		\dist(x_0,x_1)
		= \dist(x_0, R_{x_0}(\eta_0))
		\overset{\cref{eq:local-1}}{\le} c_3 \|\eta_0\|
		\le c_3 \|\eta^*\|
		\overset{\cref{eq:local-3}}{\le} c_{3}c_1 c_2 \del
		\overset{\cref{eq:local-2}}{=}
		\del_4 - \del,
	$$
	which indicates that
	$$
		\dist(x_1, x^*) \le \dist(x_1, x_0) + \dist(x_0, x^*)
		\le (\del_4 - \del) + \del = \del_4.
	$$
	Since the Algorithm~\ref{alg:sstr} is a descent algorithm, we have $x_1 \in L \cap B_{\del_4}(x^{*}) = U$. By induction, we know that $\{x_k\} \subset U$ and thus converges to the only minimizer $x^*$ by the descent property of Algorithm~\ref{alg:sstr}.
\end{proof}

\subsection{Local Convergence Rate} \label{sec:rate}

In this section, we prove the superlinear local convergence rate of our algorithm.
We first state two essential geometric laws on manifolds. 
\re{Recall that our discussion are in some normal neighborhoods.}

\begin{lemma}[Consecutive exponential maps {\cite[Proposition 15.28]{lee2012Smoothmanifolds}}]\label{lem:con-exp}
	For any two vectors $\xi_1,\xi_2 \in T_{x} \mathcal{M}$,
	we have
	$$
		\dist \left( \exp_{x}(\xi_1+\xi_2), \exp_{y}(P_{xy}\xi_2) \right) = O(\|\xi_1\|\|\xi_2\|),
	$$
	where $y = \exp _{x}(\xi_1)$.
	Or equivalently, for any $z$, we have
	$$
		\dist(\exp_{x}(\xi_1 + \xi_2), z) = \dist(\exp_{y}(P_{xy}\xi_2), z) + O(\|\xi_1\|\|\xi_2\|).
	$$
\end{lemma}
\begin{corollary}\label{lem:cosine}
For any two vectors $\xi_1,\xi_2 \in T_{x} \mathcal{M}$,
we have
$$
	\left|\dist(\exp_{x}(\xi_1),\exp _{x}(\xi_2)) - \|\xi_1 - \xi_2\|\right| = O(\|\xi_1\|\|\xi_2\|).
$$
\end{corollary}
\begin{proof}
	First, we note that
	$$
		\|\xi_1 - \xi_2\| = \dist(y, \exp _{y}(P_{xy}(\xi_2 - \xi_1))),
	$$
	where $y = \exp _{x}(\xi_1)$. Then by the triangle inequality, we get
	$$
		\begin{aligned}
			    & \left|\dist(\exp_{x}(\xi_1),\exp _{x}(\xi_2)) - \|\xi_1 - \xi_2\|\right|               \\
			=   & \left| \dist(y, \exp _{x}(\xi_2)) - \dist(y, \exp _{y}(P_{xy}(\xi_2 - \xi_1))) \right| \\
			\le & \dist(\exp _{x}(\xi_2), \exp_{y}(P_{xy}(\xi_2 - \xi_1)))                               \\
			=   & \dist(\exp _{x}(\xi_1 + (\xi_2 - \xi_1)), \exp_{y}(P_{xy}(\xi_2 - \xi_2))).            \\
		\end{aligned}
	$$
	Therefore, by \cref{lem:con-exp}, we get
	$$
		\left|\dist(\exp_{x}(\xi_1),\exp _{x}(\xi_2)) - \|\xi_1 - \xi_2\|\right| = O(\|\xi_1\|\|\xi_2 - \xi_1\|).
	$$
	Symmetrically, we have
	$$
		\left|\dist(\exp_{x}(\xi_1),\exp _{x}(\xi_2)) - \|\xi_1 - \xi_2\|\right| = O(\|\xi_2\|\|\xi_2 - \xi_1\|).
	$$
	Combining the above two equations gives
	$$
		\begin{aligned}
			    & \left|\dist(\exp_{x}(\xi_1),\exp _{x}(\xi_2)) - \|\xi_1 - \xi_2\|\right|                         \\
			\le & O\left(\min \left\{\|\xi_1\|\|\xi_2-\xi_1\|, \|\xi_2\|\|\xi_2-\xi_1\|\right\}\right)             \\
			\le & O\left(\min \left\{\|\xi_1\|\|\xi_2\|+\|\xi_1\|^2, \|\xi_1\|\|\xi_2\|+\|\xi_2\|^2\right\}\right) \\
			=   & O\left(\|\xi_1\|\|\xi_2\| + \min \left\{\|\xi_1\|^2, \|\xi_2\|^2\right\}\right)                  \\
			\le & O\left(\|\xi_1\|\|\xi_2\| + \|\xi_1\|\|\xi_2\|\right)                                            \\
			=   & O\left(\|\xi_1\|\|\xi_2\|\right).
		\end{aligned}
	$$
\end{proof}

This corollary can also be derived from the cosine law on manifolds; see \cite[Lemma 2.4]{daniilidisSelfcontractedCurvesRiemannian2018}.

Proving the superlinear local convergence rate of a trust region method requires demonstrating that the trust region will eventually become inactive.
Once this is established, one can expect the superlinear convergence rate due to the trust region step being precisely Newton's step.
However, previous work on Euclidean semismooth trust region methods either makes the inactivity of the trust region an assumption \citep{sun1997computablegeneralized}, or directly uses Newton's step near a nondegenerate local minimizer \citep{yang2000trustregion}.
In this paper, we prove the inactivity of the trust region without any additional assumptions.

The challenge in showing the inactivity of the trust region stems from the non-twice differentiability of the objective function, which makes it difficult to show that the model error is a second-order infinitesimal of the step-size, i.e., $|m(\eta) - \phi(\eta)| = o(\|\eta\|^2)$.
For twice differentiable objective functions, this condition automatically holds. However, for objective functions with only a semismooth gradient field, it is not as obvious, as the \textit{diameter} of $\pat\grad\phi(x)$ is not an infinitesimal of the step-size. That is, the disparity between the generalized Hessian in $m$ and the one in the Taylor expansion of $\phi$ can be significant in terms of operator norm.
Nonetheless, we can establish a variation of this condition for a semismooth trust region method near a nondegenerate local minimizer with a slight detour.
The key lies in recognizing that, due to its semismoothness (\cref{def:ss}), the diameter of the $\pat\grad\phi(x)$ acting on certain directions can be well controlled.
We now present the lemma on this condition.

\begin{lemma}[Model error condition]\label{lem:mod-err}
	Let $x^*$ be a nondegenerate local minimizer of $\phi$. Let $\{x_k\}\to x^*$ be a sequence generated by Algorithm~\ref{alg:sstr}. Suppose $\grad\phi$ is semismooth at $x^*$ and the retraction $R$ has a $\nu$-\holder differential near $x^{*}$. Then, near $x^*$, we have
	$$
		\left|m_{x_k}(\eta_k) - \phi(\exp _{x_k}(\eta_k))\right| = o(\|\eta_k\|(\|\eta_k\| + \|\grad\phi(x_k)\|)).
	$$
	Or equivalently,
	$$
		\lim_{k \to \infty}
		\frac{\left|m_{x_k}(\eta_k) - \phi(\exp _{x_k}(\eta_k))\right|}{\|\eta_k\|(\|\eta_k\| + \|\grad\phi(x_k)\|)} = 0.
	$$
\end{lemma}
\begin{proof}
	By \eqref{eq:taylor-4} in \cref{lem:taylor}, there exists $\hat{H}_k\in\pat\grad\phi(x_k)$ such that
	$$
		\phi(\exp_{x_k}(\eta_k)) = \phi(x_k) + \lang \grad{\phi(x_k),\eta_k}\rang + \frac{1}{2} \lang \hat{H}_k \eta_k, \eta_k\rang + o(\|\eta_k\|^2)
		.$$
	Also, at each iteration, we arbitrarily choose an $H_k \in \pat\grad\phi(x_k)$ for the model problem \eqref{eq:model}. Therefore, we have
	\begin{equation}\label{eq:mod-err-1}
		m_{x_{k}}(\eta_k) - \phi(\exp_{x_k}(\eta_k)) = \frac{1}{2}\lang(H_k - \hat{H}_k) \eta_k,\eta_k\rang + o(\|\eta_k\|^{2}).
	\end{equation}
	When $\phi$ is twice differentiable, $\pat\grad\phi(x_k)$ only contains one element, then the condition automatically holds.
	For a SC$^{1}$ function $\phi$, we want to control the diameter of $\pat\grad\phi(x_k)$ in some way.
	\re{Suppose $k$ is sufficiently large such that $x^{*}$ is in the normal neighborhood of $x_k$ and subsequent iteration points.}
	Then, we can define $\zeta_{k} = \exp_{x_k}^{-1}(x^*)$. For any $\epsilon > 0$ in Definition~\ref{def:ss}, since $\{ x_k \}$ converges to $x^*$, we know that for a sufficiently large $k$, we have
	\begin{align}\notag
		\left\| (H_{k} - \hat{H}_k)\zeta_k \right\|
		=   \Bigl\| & (\grad\phi(x^*) - P_{x_kx^*}(\grad\phi(x_k) + H_{k}\zeta_k)) -                      \\\notag
		            & \left.(\grad\phi(x^*) - P_{x_{k}x^*}(\grad\phi(x_k) + \hat{H}_{k}\zeta_k)) \right\| \\\notag
		\le \Bigl\| & \grad\phi(x^*) - P_{x_kx^*}(\grad\phi(x_k) + H_{k}\zeta_k) \Bigr\|                  \\\notag
		+           & \left\|\grad\phi(x^*) - P_{x_{k}x^*}(\grad\phi(x_k) + \hat{H}_{k}\zeta_k) \right\|  \\\label{eq:mod-err-0}
		\le 2       & \epsilon \dist(x_k,x^*)^{1+\mu}.
	\end{align}
	By letting $\epsilon$ approach zero and using $\dist(x_k,x^*)=\|\zeta_k\|$, we get
	$$
		\lim_{k \to \infty} \frac{\left\| (H_{k} - \hat{H}_k)\zeta_k \right\|}{\|\zeta_k\|} = 0,
	$$
	which can be equivalently expressed as $\|(H_k-\hat{H}_k)\zeta_k\|=o(\|\zeta_k\|)$.
	We can interpret this as the diameter of $\pat\grad\phi(x_k)$ applied on $\zeta_k$ is controled by $\|\zeta_k\|$.
	Then, by the triangle inequality, we have
	\begin{align}\notag
		\lang(H_k - \hat{H}_k) \eta_k,\eta_k\rang
		\le & \left\| (H_k - \hat{H}_k)\eta_k \right\| \|\eta_k\|                                                                  \\\notag
		\le & \left\| (H_k - \hat{H}_k)(\eta_k-\zeta_k) \right\| \|\eta_k\| + \left\| (H_k - \hat{H}_k)\zeta_k \right\| \|\eta_k\| \\\label{eq:mod-err-2}
		=   & \left\| (H_k - \hat{H}_k)(\eta_k-\zeta_k) \right\| \|\eta_k\| + o(\|\eta_k\|\|\zeta_k\|).
	\end{align}

	When $\{ x_k \}$ converges to $x^*$, the difference $\{ \dist(x_k,x_{k+1}) \}$ also shrinks to zero. Therefore, for any $\epsilon$ in item 3 of \cref{lem:cgcd}, for a sufficiently large $k$, we have
	$$
		P_{x_k x_{k+1}}\pat\grad\phi(x_k)P_{x_{k+1}x_k} \subset \pat\grad\phi(x_{k+1}) + B_{\epsilon}(0).
	$$
	That is, there exist operators $B,\hat{B} \in \mathcal{L}(\tm{x_{k+1}})$ and $A_{k+1}, \hat{A}_{k+1} \in \pat\grad\phi(x_{k+1})$ such that $\|B\|,\|\hat{B}\|\le\epsilon$ and
	$$
		P_{x_k x_{k+1}}H_{k}P_{x_{k+1}x_k} = A_{k+1} + B, \quad
		P_{x_k x_{k+1}}\hat{H}_kP_{x_{k+1}x_k} = \hat{A}_{k+1} + \hat{B}.
	$$
	Then we have
	\begin{align}\notag
		    & \left\| (H_k - \hat{H}_k)(\eta_k-\zeta_k) \right\|                                                         \\\notag
		\le & \left\| (A_{k+1} - \hat{A}_{k+1})P_{x_k x_{k+1}}(\eta_k - \zeta_{k}) \right\|
		+ \left\| (B - \hat{B})P_{x_k x_{k+1}}(\eta_k - \zeta_k) \right\|                                                \\\notag
		\le & \underbrace{\left\| (A_{k+1} - \hat{A}_{k+1})\zeta_{k+1} \right\|}_{G_1}
		+ \underbrace{\left\| (A_{k+1} - \hat{A}_{k+1})(P_{x_k x_{k+1}}(\zeta_k - \eta_k) - \zeta_{k+1}) \right\|}_{G_2} \\\label{eq:mod-err-3}
		    & + \underbrace{2 \epsilon\|\eta_k - \zeta_k\|}_{G_3},
	\end{align}
	where we used two triangle inequalities.
	By the arbitrariness of $\epsilon$, we know $G_3 = o(\|\eta_k - \zeta_k\|) = o(\|\zeta_k\| + \|\eta_k\|)$. Then similar to \eqref{eq:mod-err-0}, we have $G_1 = o(\|\zeta_{k+1}\|)$.
	We are left with $G_2$.
	By Corollary~\ref{lem:cosine}, we have
	\begin{align}\notag
		\left\|  P_{x_k x_{k+1}}(\zeta_k - \eta_k) - \zeta_{k+1}\right\| = & \dist(\exp_{x_{k+1}}(P_{x_k x_{k+1}}(\zeta_k - \eta_k)), \exp_{x_{k+1}}(\zeta_{k+1})) \\\label{eq:mod-err-4}
		                                                                   & + O (\|\zeta_k - \eta_k\|\|\zeta_{k+1}\|).
	\end{align}
	Recall that $\zeta_{k+1} = \exp _{x_{k+1}}^{-1}(x^*)$.
	By \cref{lem:con-exp}, we have
	\begin{align}\notag
		  & \dist(\exp_{x_{k+1}}(P_{x_k x_{k+1}}(\zeta_k - \eta_k)), \exp_{x_{k+1}}(\zeta_{k+1}))     \\\notag
		= & \dist(\exp_{x_{k+1}}(P_{x_k x_{k+1}}(\zeta_k - \eta_k)),x^*)                              \\\label{eq:mod-err-51}
		= & \dist(\exp _{x_k}(\eta_k' + (\zeta_k - \eta_k)),x^*) + O(\|\eta_k'\|\|\zeta_k - \eta_k\|)
	\end{align}
	where $\eta_k' = \exp^{-1}_{x_k}(x_{k+1})$.
	Applying $\zeta_k = \exp _{x_k}^{-1}(x^*)$ and \cref{lem:con-exp} again, we get
	\begin{align}\notag
		  & \dist(\exp _{x_k}(\eta_k' + (\zeta_k - \eta_k)), x^*)                                       \\\notag
		= & \dist(\exp_{x_k}(\zeta_k + (\eta_k' - \eta_k)), x^*)                                        \\\notag
		= & \dist(\exp _{x^*}(P_{x_k x^*}(\eta_k' - \eta_k)), x^*) + O(\|\zeta_k\|\|\eta_k' - \eta_k\|) \\\notag
		= & \|P_{x_kx^*}(\eta_k' - \eta_k)\| + O(\|\zeta_k\|\|\eta_k' - \eta_k\|)                       \\ \label{eq:mod-err-5}
		= & \|\eta_k' - \eta_k\| + O(\|\zeta_k\|\|\eta_k' - \eta_k\|).
	\end{align}

	Finally, by \cref{lem:cosine} and \cref{prop:r-ss}, we have
	\begin{align}\notag
		\|\eta_k' - \eta_k\|
		=   & \dist(x_{k+1}, \exp_{x_k}(\eta_k)) + O(\|\eta_k'\|\|\eta_k\|)                                           \\\notag
		\le & \dist(x_{k+1}, R_{x_k}(\eta_k)) + \dist(R_{x_k}(\eta_k), \exp _{x_k}(\eta_k))+ O(\|\eta_k'\|\|\eta_k\|) \\\label{eq:mod-err-6}
		\le & \ 0 + O(\|\eta_k\|^{1+\nu}) + O(\|\eta_k'\|\|\eta_k\|).
	\end{align}
	Recall that $\{ \dist(x_k,x_{k+1}) \}$ converges to zero. Thus, $\{ \eta_k' \}$ also converges to zero. This fact together with \eqref{eq:mod-err-6} gives
	\begin{equation}\label{eq:mod-err-7}
		\|\eta_k' - \eta_k\| = o(\|\eta_k\|) \quad \text{and} \quad
		\|\eta_k'\| = O(\|\eta_k\|).
	\end{equation}
	Combining \cref{eq:mod-err-4,eq:mod-err-51,eq:mod-err-5} and \cref{eq:mod-err-7} gives
	$$
		\begin{aligned}
			  & \left\| P_{x_k x_{k+1}}(\zeta_k - \eta_k) - \zeta_{k+1} \right\|                                                                 \\
			= & O(\|\zeta_k - \eta_k\|\|\zeta_{k+1}\|) + O(\|\eta_k'\|\|\zeta_k - \eta_k\|) + O(\|\zeta_k\|\|\eta_k' - \eta_k\|) + o(\|\eta_k\|) \\
			= & o(\|\zeta_k - \eta_k\|) + o(\|\zeta_k-\eta_k\|) + o(\|\zeta_k\|) + o(\|\eta_k\|)                                                 \\
			= & o(\|\eta_k\| + \|\zeta_k\|),
		\end{aligned}
	$$
	which further gives $\|\zeta_{k+1}\| = O(\|\eta_k\| + \|\zeta_k\|)$.
	Then for $G_1$ and $G_2$ in \eqref{eq:mod-err-3}, we have
	$$
		\begin{aligned}
			G_1 = o(\|\eta_k\| + \|\zeta_k\|), \quad
			G_2
			= \|A_{k+1} - \hat{A}_{k+1}\|\cdot o(\|\zeta_k\| + \|\eta_k\|).
		\end{aligned}
	$$
	Since $\{ x_k \}$ converges to $x^{*}$, by \cref{lem:cgcd-bound}, $\{ A_{k+1}, \hat{A}_{k+1} \}$ are uniformly bounded. Also, by \cref{lem:grad-bound}, we have
	$$\|\zeta_k\| = \|\exp_{x_k}^{-1}(x^*)\| = \dist(x_k,x^*) = O(\|\grad\phi(x_k)\|).$$ Therefore, combining \cref{eq:mod-err-1,eq:mod-err-2,eq:mod-err-3} gives
	$$
		|m_{x_k}(\eta_k) - \phi(\exp_{x_k} (\eta_k))| \le o(\|\eta_k\|(\|\eta_k\| + \|\grad\phi(x_k)\|)).
	$$
\end{proof}

\begin{remark}
	We remark that the above lemma holds for Euclidean trust region methods as the Euclidean space is also a Riemannian manifold. However, the derivation for the Euclidean case is notably more straightforward, because $G_2$ in \cref{eq:mod-err-3} vanishes as a result of $P_{x_k x_{k+1}}(\zeta_{k} - \eta_k) = \zeta_{k+1}$ in Euclidean spaces.
	Consequently, proof subsequent to \eqref{eq:mod-err-3} vanishes in the Euclidean setting.
	While on a Riemannian manifold with a general retraction, the process is much more involved due to the non-vanishing $G_2$.
\end{remark}

The next lemma shows that the trust region will eventually be inactive. For a $C^{2}$ objective function and retraction, one can easily get $|m(\eta_k) - \phi(R(\eta_k))| = O(\|\grad\phi(x_k)\|\|\eta_k\|^2)$, and then the result follows (see \cite{absil2006convergence}). However, in our setting, we need \cref{prop:r-ss,lem:mod-err} to tackle the non-twice differentiability of the objective function and retraction, respectively.
\begin{lemma}[Inactivity of trust region]\label{lem:rho}
	Let $x^*$ be a nondegenerate local minimizer of $\phi$. Let $\{x_k\}\to x^*$ be a sequence generated by Algorithm~\ref{alg:sstr}. If $\grad\phi$ is semismooth at $x^*$ and the retraction $R$ has a $\nu$-\holder differential near $x^{*}$, we have
	$$
		\lim_{k\to\infty}\rho_k = 1.
	$$
\end{lemma}

\begin{proof}
	In this proof, without confusion, we omit the subscript $x_k$ in $m_{x_k}, R_{x_k}$, and $\exp_{x_k}$.
	First by the Taylor equation~\eqref{eq:taylor-1} and \cref{prop:r-ss}, we have
	\begin{align}\notag
		|\phi(\exp(\eta_k)) - \phi(R(\eta_k))|
		=   & \left|\lang P_{\gamma}^{\tau \to 0}\grad \phi(\gamma(\tau)),\gamma'(0)\rang\right|            \\\notag
		\le & \left\|P_{\gamma}^{\tau\to0}\grad{\phi(\gamma(\tau))}\right\|  \|\gamma'(0)\|                 \\\notag
		=   & \left\|P_{\gamma}^{\tau\to0}\grad{\phi(\gamma(\tau))}\right\|  \dist(\exp (\eta_k),R(\eta_k)) \\\label{eq:rho-0}
		\le & \left\|P_{\gamma}^{\tau\to0}\grad{\phi(\gamma(\tau))}\right\|\cdot o(\|\eta_k\|),
	\end{align}
	where $\gamma$ is the geodesic from $\exp(\eta_k)$ to $R(\eta_k)$, and $\tau \in [0,1]$;
	Then we decompose the norm in \eqref{eq:rho-0}:
	\begin{align} \notag
		    & \left\|P_{\gamma}^{\tau\to0}\grad{\phi(\gamma(\tau))}\right\|                                                                                             \\\label{eq:rho-1}
		\le & \left\|P_{\gamma}^{\tau\to0}\!\grad{\phi(\gamma(\tau))}\! - \!P_{x_k\exp(\eta_k)}\grad{\phi(x_k)}\right\| \!+\! \|P_{x_k \exp (\eta_k)}\grad{\phi(x_k)}\| \\
		\label{eq:rho-2}
		\le & L\dist(\gamma(\tau), x_k) + \|\grad{\phi(x_k)}\|                                                                                                          \\
		\label{eq:rho-3}
		\le & L(\dist(\gamma(\tau),\exp(\eta_k)) + \dist(\exp (\eta_k),x_k)) + \|\grad{\phi(x_k)}\|                                                                     \\\notag
		=   & L\tau\dist(R(\eta_k),\exp(\eta_k)) + L\|\eta_k\| + \|\grad{\phi(x_k)}\|                                                                                   \\
		\label{eq:rho-4}
		=   & L\tau\cdot o(\|\eta_k\|) + L\|\eta_k\| + \|\grad{\phi(x_k)}\|,
	\end{align}
	where \eqref{eq:rho-1} and \eqref{eq:rho-3} use the triangle inequality,
	\eqref{eq:rho-2} uses the Lipschitzness of $\grad{\phi}$ and $L$ is the Lipschitz constant as in \cref{def:lip},
	and \eqref{eq:rho-4} is by \cref{prop:r-ss}.
	Combining \eqref{eq:rho-0}, \eqref{eq:rho-4}, and \cref{lem:mod-err} gives
	\begin{align} \notag
		|m(\eta_k) - \phi(R(\eta_k))| \le &
		|m(\eta_k)\!-\!\phi(\exp(\eta_k))| + |\phi(\exp (\eta_k))\!-\!\phi(R(\eta_k))|      \\\label{eq:rho-5}
		=                                 & o(\|\eta_k\|(\|\eta_k\| + \|\grad\phi(x_k)\|)).
	\end{align}

	Since $\{ x_k \}$ converges to $x^{*}$, by \cref{lem:cgcd-bound}, $\{ H_k \}$ and $\{ H_{k}^{-1} \}$ are uniformly bounded; let their operator norm uniform upper bounds be $\beta_1,\beta_2$ respectively.
	Now we denote $\zeta_k\coloneqq \grad\phi(x_k)$ and $\eta_k ^{*} \coloneqq -H_k ^{-1}\zeta_k$. By \cref{lem:tcg}, $\|\eta_k\| \le \|\eta_k ^{*}\| \le \beta_2\|\zeta_k\|$. Putting these back to \eqref{eq:rho-5} gives
	$$
		|m(\eta_k) - \phi(R(\eta_k))| = o(\|\zeta_k\|\cdot \|\eta_k\|).
	$$
	By \cref{lem:cauchy}, putting the above equation back to $\rho_k-1$ gives
	$$
		|\rho_k - 1| \le \frac{2\cdot o(\|\zeta_k\|\cdot \|\eta_k\|)}{\|\zeta_k\| \min\{ \Delta_k, \|\zeta_k\|/ \beta_1 \}}.
	$$
	When the denominator is $\Delta_k$, since $\|\eta_k\|\le \Delta_k$, we have
	$$
		|\rho_k - 1| \le \frac{2\cdot o(\|\zeta_k\|\cdot \Delta_k)}{\|\zeta_k\| \cdot \Delta_k} \to 0, \quad k\to \infty.
	$$
	Otherwise, when the denominator is $\|\zeta_k\|/\beta_1$, since $\|\eta_k\| \le \beta_2\|\zeta_k\|$, we have
	$$
		|\rho_k - 1| \le \frac{2\beta_1 \beta_2\cdot o(\|\zeta_k\|^{2})}{\|\zeta_k\|^{2}} \to 0, \quad k\to \infty.
	$$
	In conclusion, we have
	$$
		\lim_{k \to \infty}|\rho_k - 1| = 0,
	$$
	which gives $\lim_{k \to \infty}\rho_k = 1$.
\end{proof}
To the best of our knowledge, \cref{lem:rho} is the first result on the eventual inactivity of the trust region of trust region methods for SC$^1$ problems, even for the Euclidean case.
We are now ready to prove the algorithm's superlinear local convergence rate.

\begin{theorem}[Local convergence rate] \label{thm:rate}
	Let $x^*$ be a nondegenerate local minimizer of $\phi$. Let $\{x_k\}\to x^*$ be a sequence generated by Algorithm~\ref{alg:sstr}.
	If the retraction is $\nu$-order H\"older differentiable in a neighborhood of $x^*$,
	$\grad {\phi(x)}$ is $\mu$-order semismooth at $x^*$,
	Algorithm \ref{alg:tcg} uses \eqref{eq:stop} as the stopping criterion,
	then there exist $c,K>0$ such that for any $k \ge K$,
	$$
		\dist(x_{k+1}, x^*) \le c \dist(x_k,x^*)^{1+\min\{\theta,\nu,\mu\}}.
	$$

	Specifically, if $R\in C^{1,1}$, $\grad \phi(x)$ is strongly semismooth at $x^*$, and $\theta = 1$ in \eqref{eq:stop}, we have
	$$
		\dist(x_{k+1}, x^*) \le c \dist(x_k,x^*)^2.
	$$
\end{theorem}

\begin{proof}
	First, by \cref{prop:r-ss}, there exists $K_1,c_1 > 0$ such that for any $k\ge K_1$, we have
	\begin{align} \label{eq:rate-1}
		\dist(x_{k+1}, x^*)
		 & = \dist(x^*, R_{x_k}(\eta_k))\notag                                                  \\
		 & \le \dist(x^*,\exp_{x_k}(\eta_k)) + \dist(\exp_{x_k}(\eta_k), R_{x_k}(\eta_k))\notag \\
		 & \le \dist(x^*,\exp_{x_k}(\eta_k)) + c_1\|\eta_k\|^{1+\nu}.
	\end{align}

	\re{Suppose $K_1$ is sufficiently large such that $x^{*}$ is in the normal neighborhood of $x_k$ and subsequent iteration points.}
	Then, we can define $\zeta_k \coloneqq \exp_{x_k}^{-1}(x^*)$. By Corollary~\ref{lem:cosine}, there exists $c_2 > 0$ such that
	\begin{equation}\label{eq:rate-2}
		\dist(x^*,\exp_{x_k}(\eta_k))
		= \dist(\exp_{x_k}(\zeta_k),\exp_{x_k}(\eta_k))
		\le \|\eta_k - \zeta_k\| + c_2\|\zeta_k\|\|\eta_k\|.
	\end{equation}

	Let $\eta_{k}^* \coloneqq -H_k^{-1}\grad \phi(x_k)$. Then we have
	\begin{align}\label{eq:rate-3}
		\|\eta_k - \zeta_k\| & = \|\eta_k - \eta_k^* + \eta_k^* - \exp_{x_k}^{-1}(x^*)\|\notag                 \\
		                     & \le \|\eta_k - \eta_k^*\| + \|H_k^{-1}\grad \phi(x_k) + \exp_{x_k}^{-1}(x^*)\|.
	\end{align}

	By \cref{lem:cgcd-bound}, there exists $K_2 > 0$ and $\beta_2 > 0$ such that $\{H_k^{-1}\}_{k\ge K_2}$ is uniformly bounded by $\beta_2$. Also, by the semismoothness of $\grad \phi(x)$ at $x^*$, there exists $K_3> 0$ such that for any $k\ge K_3$,
	\begin{align}\label{eq:rate-4}
		\|H_k^{-1}\grad \phi(x_k) + \exp_{x_k}^{-1}(x^*)\|
		 & \le \|H_k^{-1}\|\|\grad \phi(x_k) + H_k\exp_{x_k}^{-1}(x^*)\|\notag \\
		 & \le \beta_2 \dist(x_k,x^*)^{1+\mu}.
	\end{align}

	At last, we need to bound $\|\eta_k -\eta_k^*\|$. By \cref{lem:rho}, there exists $K_4 > 0$ such that for any $k\ge K_4$, the trust region is inactive. Then the trust region radius will reach the radius cap set in Algorithm~\ref{alg:sstr}, i.e., $\Delta_k = \bar{\Delta} > 0$. On the other hand, by \cref{lem:tcg} we have
	$$
		\lim_{k \to \infty}\|\eta_k\| \le \lim_{k \to \infty}\|\eta_k^*\| \le \lim_{k \to \infty}\|H_{k}^{-1}\|\|\grad\phi(x_k)\| \le \beta_2\lim_{k \to \infty}\|\grad\phi(x_k)\| = 0.
	$$
	Therefore, the second truncation condition in Algorithm \ref{alg:tcg} (line 8) shall not be met.
	Moreover, by \cref{lem:cgcd-bound}, the elements in $\pat\grad\phi(x)$ when $x$ is near $x^{*}$ are positive definite (say $k\ge K_4$). Therefore, the first truncation in Algorithm~\ref{alg:tcg} (line 3) shall not be met.
	All in all, for $k\ge K_4$, Algorithm \ref{alg:tcg} terminates using \eqref{eq:stop}, making the final residual satisfy $r_j = r_0 + H_k\eta_k = \grad \phi(x_k) + H_k\eta_k$. Therefore,
	$$
		\|\eta_k - \eta_k^*\| = \|H_k^{-1}(H_k\eta_k + \grad \phi(x_k))\| = \|H_k^{-1}r_j\| \overset{\cref{eq:stop}}{\le} \beta_2 \|r_0\|^{1+\theta},
	$$
	where $r_0 = \grad\phi(x_k)$ and $\theta > 0$ are set in Algorithm~\ref{alg:tcg}. Again, by \cref{lem:cgcd-bound}, let $\beta_1$ be the operator norm uniform upper bound of $\{ H_k \}$ when $k\ge K_4$. By \cref{lem:grad-bound}, we have
	\begin{equation}\label{eq:rate-5}
		\|\eta_k - \eta_k^*\| \le \beta_2 \|\grad \phi(x_k)\|^{1+\theta} \le \beta_2 \beta_1^{1+\theta}\dist(x^*,x_k)^{1+\theta}.
	\end{equation}

	Let $K = \max\{K_1,K_2,K_3,K_4\}$. Combining \cref{eq:rate-1,eq:rate-2,eq:rate-3,eq:rate-4,eq:rate-5} gives
	\begin{align}\notag
		\dist(x_{k+1},x^*) \le & c_1\|\eta_k\|^{1+\nu} + c_2\|\eta_k\|\dist(x_k,x^*)                         \\\label{eq:rate-6}
		                       & + \beta_2(1 + \beta_1^{1+\theta})\dist(x_k,x^*)^{\min{\{1+\theta,1+\mu\}}}.
	\end{align}
	Similar to \eqref{eq:rate-5}, we also have $\|\eta_k\| \le \|\eta_k^*\| \le \beta_2 \|\grad\phi(x_k)\| \le \beta_2 \beta_1 \dist(x_k,x^*)$.
	Plugging this back into \eqref{eq:rate-6} gives
	$$
		\dist(x_{k+1},x^*) \le \beta_2\left( c_1 \beta_{2}^{\nu} \beta_1^{1+\nu} + c_2\beta_1 + (1 + \beta_1^{1+\theta}) \right) \dist(x_k,x^*)^{1+\min{\{\theta,\nu,\mu\}}}.
	$$

	Finally, letting $c = \beta_2\left( c_1 \beta_{2}^{\nu} \beta_1^{1+\nu} + c_2\beta_1 + (1 + \beta_1^{1+\theta}) \right)$, we get the result.
\end{proof}

\section{Application: Solving Augmented Lagrangian Method Subproblem} \label{sec:app}
Our primary motivation for proposing a semismooth Riemannian trust region method is to develop a new technique tailored for solving the subproblem of an augmented Lagrangian method (ALM) on manifolds.
In this section, we briefly review ALM on manifolds and formulate its subproblem in the form of \eqref{eq:problem}.
Then, in \cref{sec:exp}, we assess the performance of Algorithm~\ref{alg:sstr} as the subproblem solver for ALM on manifolds.

We consider the following optimization problem
$$
\min_{x\in\M} f(x) + g(h_1(x)) \quad\text{s.t.}~ h_2(x) \le 0,
$$
where $\M$ is a Riemannian manifold, $f,h_1,h_2$ are continuously differentiable on $\M$ with Lipschitz continuous differentials, and $g$ is convex and lower semicontinuous.
Recently, \cite{kangkang2022inexact} and \cite{zhou2022semismooth} proposed two manifold inexact augmented Lagrangian methods to tackle this problem with convergence guarantee.
By introducing two new variables, we get the reformulation:
\begin{equation}\notag
	\begin{aligned}
		\min_{x,y,z} \quad & f(x) + g(y) + \delta_{\R_-^n}(z)   \\
		\text{s.t.}\quad   & x\in \M,\ y = h_1(x),\ z = h_2(x),
	\end{aligned}
\end{equation}
where $\dr(z)$ is the indicator function which equals 0 if $z\le 0$ and $+\infty$ otherwise, replacing the inequality constraint.
Then the augmented Lagrangian function of the reformulation is
\begin{align*}\notag
	L_\sigma(x,y,z,\lambda,\gamma)
	= & f(x) + g(y) + \dr(z) + \lambda^T(h_1(x) - y) + \gamma^T(h_2(x) - z)                                                         \\
	  & + \frac{\sigma}{2}\|h_1(x) - y\|_2^2 + \frac{\sigma}{2}\|h_2(x) - z\|_2^2\notag                                             \\
	= & f(x) + g(y) + \dr(z) + \frac{\sigma}{2}\left\|h_1(x) - y + \frac{\lambda}{\sigma}\right\|_2^2                               \\
	  & + \frac{\sigma}{2}\left\|h_2(x) - z + \frac{\gamma}{\sigma}\right\|_2^2 - \frac{\|\lambda\|_2^2 + \|\gamma\|^2_2}{2\sigma}.
\end{align*}

The idea of ALM is to solve the minimization problem of $L_\sigma$ with respect to $x$, $y$, and $z$ respectively in each step. We then use the Moreau envelope to further simplify the subproblem. Using partial minimization, we have
\begin{align*}
	\min_yL_\sigma & = \min_y g(y) + \frac{\sigma}{2}\left\|h_1(x)+ \frac{\lambda}{\sigma} - y\right\|_2^2 = M_{g}^{\sigma}\left( h_1(x) + \frac{\lambda}{\sigma} \right),   \\
	\min_zL_\sigma & = \min_z \dr(z) + \frac{\sigma}{2}\left\|h_2(x)+ \frac{\gamma}{\sigma} - z\right\|_2^2 = M_{\dr}^{\sigma}\left( h_2(x) + \frac{\gamma}{\sigma} \right),
\end{align*}
where $M_{g}^{\sigma}$ and $M_{\dr}^{\sigma}$ are Moreau envelopes defined as follows:
$$
	M^{\sigma}_\psi(u) \coloneqq \min_x \left\{\psi(x) + \frac{\sigma}{2}\|u - x\|_2^2\right\}.
$$
Once the optimal $x$ is given, these two subproblems (the $y$-subproblem and the $z$-subproblem) can be directly solved using proximal operators, and their solutions are given by
\begin{align*}
	y & = \argmin_yL_\sigma = \prox_{g/\sigma}\left( h_1(x) + \frac{\lambda}{\sigma}\right),                                                               \\
	z & = \argmin_zL_\sigma = \prox_{\dr/\sigma}\left( h_2(x) + \frac{\gamma}{\sigma}\right) = \proj_{\R_-^n}\left( h_2(x) + \frac{\gamma}{\sigma}\right).
\end{align*}
Then the problem $\min_{x,y,z}L_\sigma$ is equivalent to solving the following problem:
\begin{align}\label{eq:x-sub}
	x =
	\underset{x\in\M}{\argmin}\left\{ \varphi(x,\sigma,\lambda,\gamma)
	\coloneqq
	f(x) + M_g^\sigma \left( h_1(x) + \frac{\lambda}{\sigma} \right) + M_{\dr}^\sigma \left( h_2(x) + \frac{\lambda}{\sigma}\right)\right\}.
\end{align}
Since Moreau envelopes $M_g^\sigma$ and $M_{\dr}^\sigma$ have Lipschitz continuous gradient but may not necessarily be twice differentiable,  the $x$-subproblem \eqref{eq:x-sub} is captured by \eqref{eq:problem}, and is the most difficult part of the problem. Hence, we usually refer to problem \eqref{eq:x-sub} by the subproblem of ALM.
For reference, we present the ALM using our trust region method as a subproblem solver in Algorithm~\ref{alg:alm}.
We also remark that both $M_g^\sigma$ composed with $h_1$ and $M_{\dr}^\sigma$ composed with $h_2$ have semismooth gradient field on manifiold $\M$.

\begin{remark}
	For objective functions of the form $\phi = \sum_{k=1}^{K} \phi_{k}$, we can use $\mathcal{K}\coloneqq\bigcup_{k=1}^{K}\pat\grad \phi_{k}$ in place of $\pat\grad \phi$ as the generalized Hessian, if the former is easier to compute. Notably, if $\grad \phi_k$ is semismooth w.r.t. $\pat\grad\phi_k$, then $\grad \phi$ is semismooth w.r.t. $\mathcal{K}\supset \pat\grad\phi$ \citep[Proposition 4.4]{zhou2022semismooth}.
	For the results in \cref{sec:local,sec:rate} to hold using $\mathcal{K}$, we need to assume all elements in $\mathcal{K}(x^{*})$ are positive definite for a nondegenerate local minimizer $x^{*}$, then the proof falls through similarly.
\end{remark}

\cite{kangkang2022inexact}  solves the subproblem using the Riemannian gradient descent method, while \cite{zhou2022semismooth}  employs a semismooth Newton method that falls back to the gradient descent method when encountering negative curvatures.  It is worth noting that neither of them incorporates a full second-order method to solve the subproblem in the ALM, whereas our trust region method does.

\begin{algorithm}[!th]
	\caption{Manifold Inexact Augmented Lagrangian Framework}\label{alg:alm}
	\Input{$x_0 \in \M, \lambda_0 \in \R^m, \gamma_0\in\R^n_+, \sigma_0 > 0, \alpha, \tau\in(0,1), \kappa > 1, \{\epsilon_k\}\subset{\R_+}$ converging to $0$. Initialize $y_0 = \prox_{g/\sigma_0}(h_1(x_0) + \lambda_0/\sigma_0)$,
	$z_0 = \proj_{\R_-^n}(h_2(x_0) + \gamma_0/\sigma_0)$.
		Choose a constant $\bar{L} \ge \max\{L_{\sigma_0}(x_0,y_0,z_0,\lambda_0,\gamma_0), f(\bar{x}) + g(h_1(\bar{x}))\}$, where $\bar{x}\in\M$ is any feasible point.}\;
		\For {$k = 1,2,\dots$}{
		Inexactly solve the $x$-subproblem \eqref{eq:x-sub}: find $x_k \in\M$ such that $\|\grad\varphi_k(x_k)\|\le \epsilon_k$ and $\varphi_k(x_k)\le \bar{L}$, where $\varphi_k(\cdot) \coloneqq \varphi(\cdot,\sigma_k,\lambda_k,\sigma_k)$.
		Then update $y_k = \prox_{g/\sigma_k}(h_1(x_k) + \lambda_k/\sigma_k)$ and $z_k = \proj_{\R_-^n}(h_2(x_k) + \gamma_k/\sigma_k)$.\;
		Update $\lambda_{k+1} = \lambda_k + \sigma_k(h_1(x_k) - y_k)$ and $\gamma_{k+1} = \gamma_k + \sigma_k(h_2(x_k) - z_k)$.\;
		Update $\delta_k = \max\{\|h_1(x_k)-y_k\|_2,\|h_2(x_k)-z_k\|_2\}$.\;
		\uIf {$\delta_k \le \tau \delta_{k-1}$} {
			$\sigma_{k+1} = \sigma_k$.\;
		} \Else {
	$\sigma_{k+1} = \max\{\kappa\sigma_k, \|\lambda_{k+1}\|_2^{1 + \alpha}, \|\gamma_{k+1}\|_2^{1 + \alpha}\}$.\;
	}
	}
\end{algorithm}

\section{Numerical Experiments} \label{sec:exp}

In this paper, we conducted a performance evaluation of our algorithm on two distinct problem domains: compressed modes (CM) \citep{ozolicnvs2013compressed} and sparse principal component analysis (SPCA) \citep{zou2006sparse}. To assess the effectiveness of our approach, we compared our results against several state-of-the-art algorithms: SOC \citep{lai2014splitting}, ManPG \cite[ManPG-Ada (Algorithm 2)]{chen2020proximal}, ALMSSN \citep{zhou2022semismooth}, accelerated ManPG (AManPG) \citep{huang2022extension}, accelerated Riemannian proximal gradient (ARPG) \citep{huang2021riemannian}, and manifold inexact augmented Lagrangian method (MIALM) \citep{kangkang2022inexact}.

Our algorithm was implemented using the manopt package\footnote{\href{https://www.manopt.org/index.html}{\texttt{https://www.manopt.org/index.html}}} in MATLAB and executed on a standard PC equipped with an AMD Ryzen 7 5800H with Radeon Graphics CPU and 16GB RAM. In our experimental results, we use ``ALMSRTR'' to represent our algorithm. To ensure statistical significance, we conducted 20 independent instances for each parameter setting and reported the average experimental results.

\subsection{Compressed Modes in Physics}
The compressed modes (CM) problem is a mathematical physics problem that focuses on obtaining sparse solutions for a specific class of problems, such as the Schrödinger equation in quantum mechanics. To induce sparsity, the wave function undergoes $L_1$ regularization, leading to compact support solutions known as compressed modes.

Following the framework proposed by \citep{chen2020proximal}, the CM problem can be formulated as
\[
	\min_{P \in \mathrm{St}(n, r)} \left\{ \mathrm{tr}\left(P^{T} H P\right) + \mu \|P\|_{1} \right\},
\]
where $\mathrm{St}(n,r) := \{P \in \mathbb{R}^{n \times r} : P^{T}P = I_r \}$ is the Stiefel manifold and $\mu$ is a regularization parameter.
For additional details, readers may refer to \citep{ozolicnvs2013compressed}. Our experimental setup was consistent with that of \citep{zhou2022semismooth}.
In our experiments, we employ the retraction described in Algorithm \ref{alg:sstr}, which utilizes the QR decomposition method introduced in \cite[Example 4.1.3]{absilOptimizationAlgorithmsMatrix2008}.
Additionally, to maintain consistency, we employ QR decomposition in ALMSSN in the subsequent experiments.

\paragraph{Implementation details.} It is worth noting that SOC considers an equivalent problem, which can be expressed as follows:
\begin{equation*}
	\begin{aligned}
		\min_{X, Q, P \in \mathbb{R}^{n \times r}} & \operatorname{tr}\left(P^{T} H P\right) + \mu\|X\|_1 \\
		\text{s.t.} \quad                          & X=P, Q=P, Q^{T} Q=I_r .
	\end{aligned}
\end{equation*}
The Lagrangian of the equivalent problem is given by:
\[
	L_S(Q, P, X, \Gamma, \Lambda) = \operatorname{tr}\left(P^{T} H P\right) + \mu\|X\|_1 + \operatorname{tr}\left(\Gamma^{T}(Q-P)\right) + \operatorname{tr}\left(\Lambda^{T}(X-P)\right),
\]
where $P, X \in \mathbb{R}^{n\times r}$ and $Q \in \mathrm{St}(n, r)$. Therefore, the termination conditions for SOC are as follows:
\begin{equation*}
	\begin{aligned}
		 & \frac{\|Q-P\|_{\infty}}{\max \left\{\|Q\|_F,\|P\|_F\right\}+1} + \frac{\|X-P\|_{\infty}}{\max \left\{\|X\|_F,\|P\|_F\right\}+1} \leq 5 \times 10^{-7},                                                             \\
		 & \frac{\left\|\operatorname{grad}_Q L_S\right\|_{\infty}}{\|Q\|_F+1} + \frac{\left\|\nabla_P L_S\right\|_{\infty}}{\|P\|_F+1} + \frac{\min _{G \in \partial_X L_S}\|G\|_{\infty}}{\|X\|_F+1} \leq 5 \times 10^{-5}.
	\end{aligned}
\end{equation*}
Similar to MIALM and ALMSSN, our algorithm translates the problem to
\begin{equation*}
	\begin{gathered}
		\min _{P, Q \in \mathbb{R}^{n \times r}}\left\{ \operatorname{tr}\left(P^{T} H P\right) + \mu\|Q\|_1 \right\}, \\
		\text{s.t. } P=Q, P \in \operatorname{St}(n, r) .
	\end{gathered}
\end{equation*}
The Lagrangian is given by $L_N(P, Q, \Lambda) = \operatorname{tr}\left(P^{T} H P\right) + \mu\|Q\|_1 + \operatorname{tr}\left(\Lambda^{T}(P-Q)\right)$, where $Q \in \mathbb{R}^{n \times r}$ and $P \in \operatorname{St}(n, r)$.
We terminate these three algorithms when the following conditions hold:
\begin{equation}
	\label{tc for us}
	\begin{aligned}
		\frac{\|P-Q\|_{\infty}}{\max \left\{\|P\|_F,\|Q\|_F\right\}+1} \leq 5 \times 10^{-7}, \\
		\frac{\left\|\operatorname{grad}_P L_N\right\|_{\infty}}{\|P\|_F+1} + \frac{\min _{G \in \partial_Q L_N}\|G\|_{\infty}}{\|Q\|_F+1} \leq 5 \times 10^{-5}.
	\end{aligned}
\end{equation}
Following the same termination criterion as \citep{zhou2022semismooth}, we terminate ManPG, AManPG, and ARPG when $t^{-1}\left\|V_*\right\|_{\infty} /\left(\|P\|_F+1\right) \leq 5 \times 10^{-5}$, where
\[
	V_* = \underset{V \in T_P \operatorname{St}(n, r)}{\arg \min}\left\{\left\langle \operatorname{grad}_P \operatorname{tr}\left(P^{T} H P\right), V\right\rangle + \frac{1}{2 t}\|V\|_F^2 + \|P+V\|_1\right\}.
\]
We adopt the notation $\mathbf{P}_P V$ to represent the projection of a vector $V \in \mathbb{R}^{n \times r}$ onto $T_P \operatorname{St}(n, r)$. Since $-V_* / t \in 2 \mathbf{P}_P(HP) + \mu \mathbf{P}_P\left(\partial\left\|P+V_*\right\|_1\right)$, this termination criterion serves as an approximation of the first-order optimality condition for the CM problem, namely, the condition $0 \in 2 \mathbf{P}_P(HP) + \mu \mathbf{P}_P\left(\partial\left\|P\right\|_1\right)$.

The implementations of ManPG, AManPG, ARPG, and SOC in our study are consistent with the approaches employed in \citep{zhou2022semismooth}. We directly utilize the same codes and parameters as described in that particular work. Similarly, for MIALM, we adopt the codes and parameters provided in \citep{kangkang2022inexact}, while the codes and parameters of ALMSSN are based on the specifications presented in \citep{zhou2022semismooth}. In our experimentation, we establish termination criteria based on predefined conditions, and additionally, we terminate all six methods if the number of iterations exceeds 30,000.

In Algorithm \ref{alg:alm}, the sequence ${\epsilon_k}$ is set to $\epsilon_k = 0.8^k$, and the initial value of $\sigma_0$ is set to $1$. The dual variable $\lambda_0$ is initialized as a null matrix, while the primal variable $x_0$ is randomly generated using the same procedure as the other algorithms. The parameter $\tau$ is set to $0.99$, and the initial value of the parameter $\kappa$ is chosen as $1.25$. The initial maximum number of iterations in Algorithm~\ref{alg:sstr} is set to $40$ for $n \geq 500$ and $60$ for $n < 500$, where $n$ represents the data dimension. Both the maximum number of iterations of Algorithm~\ref{alg:sstr} and the value of $\kappa$ are adaptively adjusted during the iteration based on the accuracy of the current iteration. The maximum number of iterations for Algorithm~\ref{alg:tcg} used to solve the model problem in Algorithm~\ref{alg:sstr} is set to $300$. In Algorithm~\ref{alg:sstr}, we set $\bar{\Delta} = 10$, $\Delta_0 = 0.01$, and $\rho' = 0.1$. During our experiments, we noticed that ManPG exhibited notably slow performance when $n=500$, which led us to implement a 120-second time limit for the algorithm's termination.

We present the results of our experiments in Table \ref{tab:cm} and visualize them in Figure \ref{fig:cm}. The findings in Table \ref{tab:cm} indicate that all algorithms achieve similar objective function values across different settings. However, in the majority of cases, our proposed method demonstrates superior computational efficiency, suggesting its practical advantage. From Figure \ref{fig:cm}, we observe that first-order algorithms encounter limitations in achieving the desired termination condition, particularly when $r$ is large. In contrast, our method, which consistently leverages second-order information, exhibits improved convergence compared to both first-order and second-order algorithms that do not consistently utilize second-order information, even for large $r$. This underscores the superiority of our method, which relies on the consistent utilization of second-order information.

\begin{table}[ht]%{c}
	\caption{~Comparison of CM. Significant results are presented in bold. The setting for $(n, r, \mu) = (1000, 20, 0.1)$ remains fixed, while one of the dimensions varies. The ManPG approach refers to the adaptive version proposed by \cite{chen2020proximal} (ManPG-Ada). LS-\uppercase\expandafter{\romannumeral1} denotes the ALMSSN technique with the first line-search method introduced by \cite{zhou2022semismooth}, while LS-\uppercase\expandafter{\romannumeral2} utilizes the second line-search method proposed by \cite{zhou2022semismooth}.}
	\label{tab:cm}%\\
	\centering
	\begin{threeparttable}
		\setlength\tabcolsep{3mm}
		\resizebox{\linewidth}{!}
		% \scriptsize
		{\begin{tabular}{@{}cccccccccl@{}}
				\toprule
				      &      & ManPG  & AmanPG & ARPG  & SOC   & MIALM         & LS-\uppercase\expandafter{\romannumeral1} & LS-\uppercase\expandafter{\romannumeral2} & ALMSRTR        \\ \midrule
				\multicolumn{8}{l}{Running time (s)}                                                                                                                                    \\
				$n$   & 200  & 23.61  & 5.76   & 9.16  & 2.91  & 20.93         & 1.58                                      & 2.06                                      & \textbf{1.30}  \\
				      & 500  & 115.56 & 13.95  & 13.27 & 7.15  & 24.11         & 5.05                                      & 7.61                                      & \textbf{4.80}  \\
				      & 1000 & 69.03  & 12.70  & 15.59 & 22.68 & 12.11         & 14.19                                     & 12.97                                     & \textbf{9.71}  \\
				      & 1500 & 62.00  & 59.74  & 50.38 & 48.56 & 70.28         & 29.28                                     & 34.99                                     & \textbf{21.69} \\
				      & 2000 & 42.02  & 46.15  & 49.50 & 46.71 & 73.64         & 29.56                                     & 26.98                                     & \textbf{20.03} \\
				$r$   & 10   & 14.64  & 13.88  & 13.28 & 16.60 & 30.24         & 51.76                                     & \textbf{8.96}                             & 11.61          \\
				      & 15   & 33.49  & 10.54  & 12.93 & 14.95 & \textbf{8.39} & 47.95                                     & 10.39                                     & \textbf{8.39}  \\
				      & 25   & 91.43  & 23.43  & 23.48 & 26.61 & 54.02         & 15.36                                     & 17.48                                     & \textbf{14.00} \\
				      & 30   & 86.31  & 28.11  & 27.07 & 22.09 & 29.51         & \textbf{10.18}                            & 13.06                                     & 13.35          \\
				$\mu$ & 0.05 & 93.51  & 16.05  & 14.16 & 6.72  & 8.74          & \textbf{6.06}                             & 6.72                                      & 7.07           \\
				      & 0.15 & 57.12  & 18.55  & 21.45 & 31.57 & 25.97         & 12.97                                     & 14.41                                     & \textbf{11.84} \\
				      & 0.20 & 51.48  & 13.70  & 21.62 & 34.21 & 23.88         & 19.24                                     & 19.73                                     & \textbf{10.68} \\
				      & 0.25 & 36.41  & 12.65  & 22.14 & 35.83 & 13.48         & 53.62                                     & 30.18                                     & \textbf{12.47} \\
				\multicolumn{8}{l}{Loss function value}                                                                                                                                 \\
				%:$\left\{\mathrm{tr}\left(P^{T} H P\right)+\mu\|P\|_{1}\right\}$}                                      \\
				$n$   & 200  & 14.18  & 14.18  & 14.18 & 14.18 & 14.18         & 14.17                                     & 14.17                                     & 14.16          \\
				      & 500  & 18.63  & 18.63  & 18.63 & 18.63 & 18.63         & 18.63                                     & 18.63                                     & 18.63          \\
				      & 1000 & 23.37  & 23.36  & 23.36 & 23.36 & 23.36         & 23.36                                     & 23.36                                     & 23.36          \\
				      & 1500 & 26.97  & 26.86  & 26.86 & 26.86 & 26.86         & 26.86                                     & 26.86                                     & 26.86          \\
				      & 2000 & 29.98  & 29.74  & 29.74 & 29.74 & 29.74         & 29.74                                     & 29.74                                     & 29.74          \\
				$r$   & 10   & 10.77  & 10.74  & 10.74 & 10.74 & 10.74         & 10.74                                     & 10.74                                     & 10.74          \\
				      & 15   & 16.51  & 16.46  & 16.46 & 16.46 & 16.46         & 16.46                                     & 16.46                                     & 16.46          \\
				      & 25   & 32.01  & 32.00  & 32.00 & 32.00 & 32.00         & 32.00                                     & 32.00                                     & 32.00          \\
				      & 30   & 42.95  & 42.93  & 42.93 & 42.94 & 42.94         & 42.92                                     & 42.92                                     & 42.93          \\
				$\mu$ & 0.05 & 15.15  & 15.14  & 15.14 & 15.14 & 15.14         & 15.14                                     & 15.14                                     & 15.14          \\
				      & 0.15 & 31.05  & 31.01  & 31.01 & 31.01 & 31.01         & 31.01                                     & 31.01                                     & 31.01          \\
				      & 0.20 & 38.32  & 38.27  & 38.27 & 38.28 & 38.27         & 38.27                                     & 38.27                                     & 38.27          \\
				      & 0.25 & 45.36  & 45.26  & 45.26 & 45.29 & 45.26         & 45.26                                     & 45.26                                     & 45.26          \\ \bottomrule
			\end{tabular}}
	\end{threeparttable}
\end{table}

\begin{figure}[ht]
	\centering
	\subfigbottomskip=2pt
	\subfigcapskip=-5pt
	\subfigure{
		\includegraphics[width=0.32\linewidth]{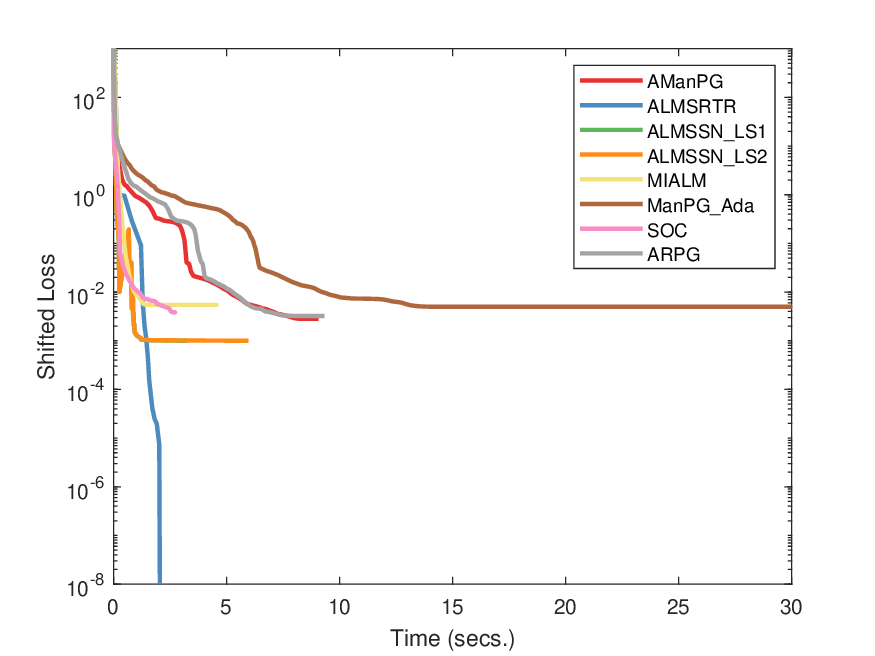}}
	\subfigure{
		\includegraphics[width=0.32\linewidth]{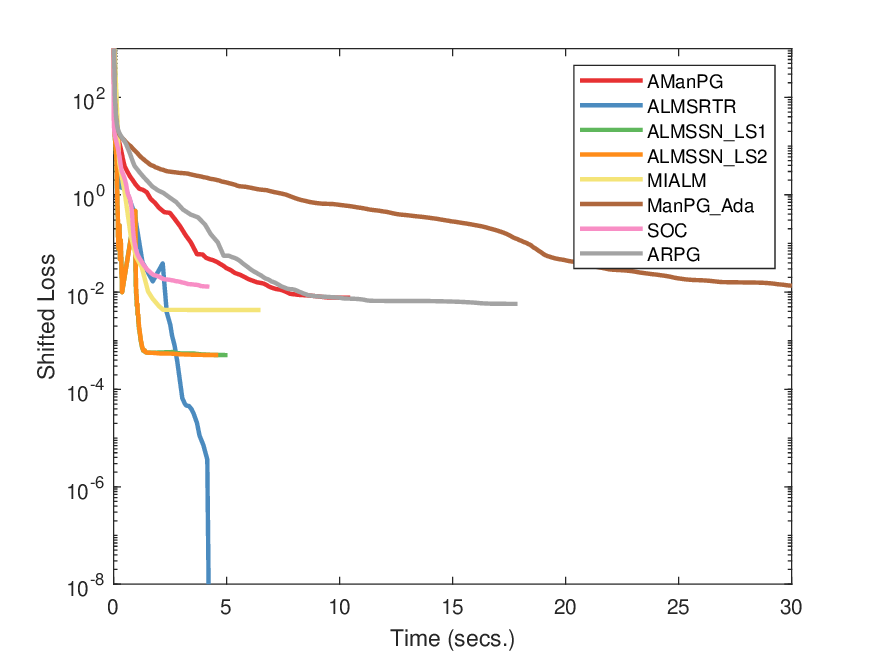}}
	\subfigure{
		\includegraphics[width=0.32\linewidth]{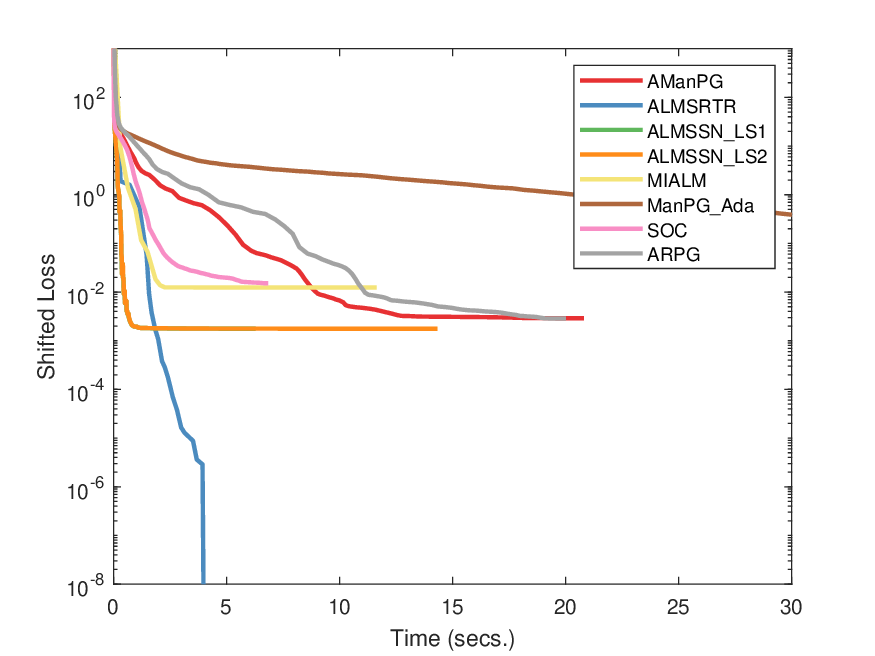}}
	\\
	\subfigure{
		\includegraphics[width=0.32\linewidth]{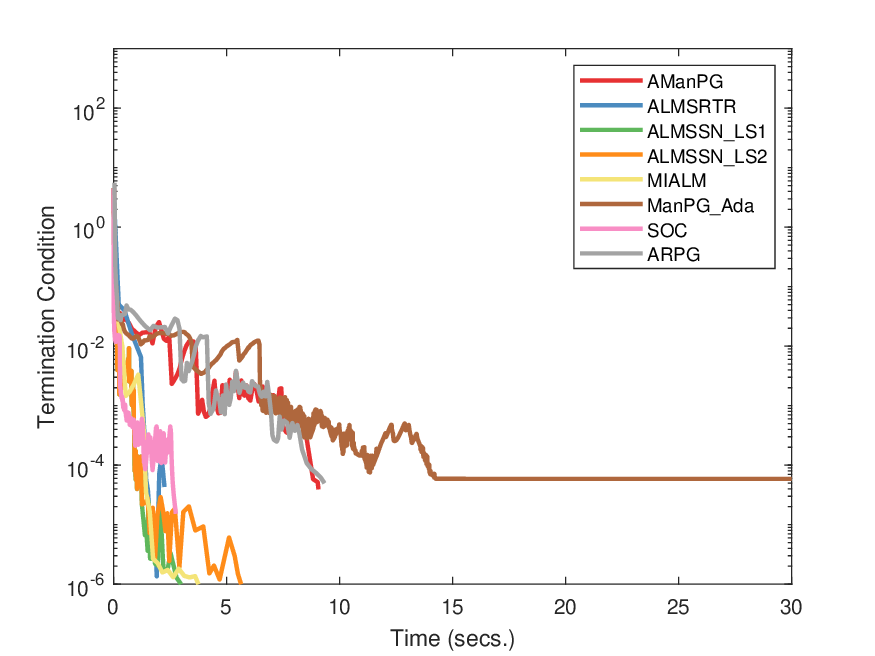}}
	\subfigure{
		\includegraphics[width=0.32\linewidth]{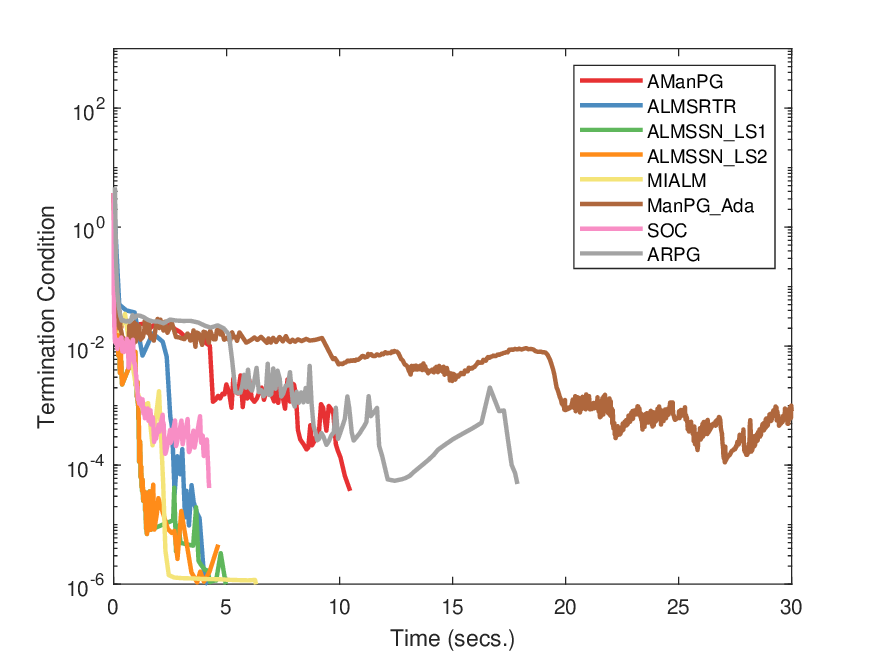}}
	\subfigure{
		\includegraphics[width=0.32\linewidth]{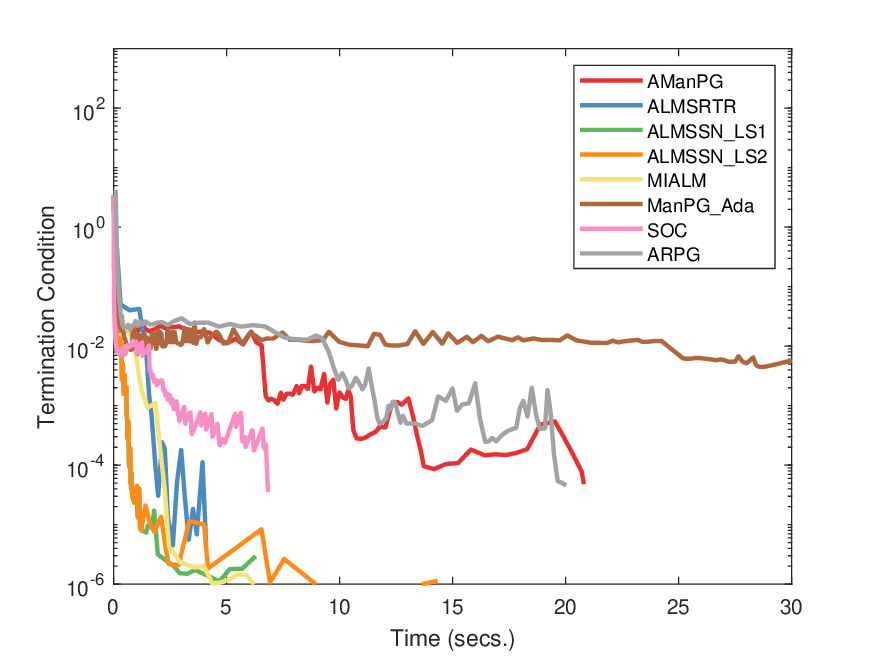}}
	\caption{\ \ Comparisons of CM with $(n, \mu) = (500, 0.05)$ and $r = 20, 30, 40$. The shifted loss represents the difference between the loss function value $F_k$ and the minimum of the values obtained at the termination of each algorithm. The second line displays the termination conditions, which are the sum of the left-hand components in \eqref{tc for us} for our algorithm and ALMSSN, as well as their respective counterparts for the other algorithms, plotted against time. ALMSSN\_LS1 denotes the ALMSSN technique with the first line-search method introduced by \cite{zhou2022semismooth}, while ALMSSN\_LS2 utilizes the second line-search method proposed by \cite{zhou2022semismooth}.}
	\label{fig:cm}
\end{figure}

\subsection{Sparse Principal Component Analysis}

This section presents the results of our experiments on the sparse principal component analysis (SPCA) problem.
SPCA is a widely used technique in data analysis that offers better interpretability compared to traditional principal component analysis. It achieves this by incorporating lasso regularization, which produces modified principal components with sparse loadings. The SPCA problem can be formulated as follows:
\begin{equation*}
	\min_{P \in \mathrm{St}(n, r)}\left\{-\mathrm{tr}\left(P^{T} A^{T} A P\right)+\mu\Vert P\Vert_{1}\right\},
\end{equation*}
where $\mathrm{St}(n,r) := \{P \in \mathbb{R}^{n \times r}:P^{T}P = I_r \}$ represents the Stiefel manifold, and $A \in \mathbb{R}^{p\times n}$ is the data matrix with $p$ observations and $n$ variables.

To evaluate the efficacy of our proposed algorithm, we conducted a comparative evaluation against ALMSSN, AManPG, ARPG, and SOC, utilizing a high level of accuracy. To ensure reliable convergence, we set the termination threshold at $5 \times 10^{-8}$ for synthetic data and $5 \times 10^{-7}$ for real data.
The termination condition remains consistent with the preceding section, with the exception that we substitute $H$ with $-A^{T}A$.

Due to not meeting our accuracy requirements or exhibiting excessively long runtimes, we excluded ManPG and MIALM from our experiments. The parameters and codes used for ALMSSN are consistent with those reported in \citep{zhou2022semismooth}.
To ensure a high level of accuracy, we incorporated several enhancements into the AManPG and ARPG codes. Our primary modification focused on refining the subproblem solution, aiming to improve overall accuracy. Furthermore, we introduced specific modifications to the AManPG code to address situations where its accuracy remains almost unchanged after a certain iteration number. The parameters employed for AManPG aligned with those documented in \citep{huang2022extension}, while the parameters for ARPG corresponded to those specified in \citep{huang2021riemannian}. For SOC, we maintained consistency by employing the parameters and code outlined in \citep{zhou2022semismooth}.

In Algorithm \ref{alg:alm}, we set $\epsilon_k=0.95^k$ and $\tau=0.9$. The initial value of $\kappa$ is set to $1.25$. In Algorithm \ref{alg:sstr}, the maximum number of iterations is set to $50$ when $n < 2000$, $70$ when $2000 \leq n < 3000$, and $90$ when $n \geq 3000$. Both the maximum number of iterations and the value of $\kappa$ are adaptively adjusted during each iteration based on the accuracy of the current iteration. The remaining parameters in our algorithm are consistent with those in the CM algorithm.

The data matrix $A$ is generated using two different methods:

\begin{enumerate}
	\item[(1)] \textbf{Synthetic.}
	      The data matrix $A$ is randomly generated using the method described in \citep{zhou2022semismooth}. Various ill-conditioned matrices $A \in \mathbb{R}^{50\times n}$ can be obtained in different dimensions.
	\item[(2)] \textbf{Real.}
	      The data matrix $A$ is selected from real datasets. The gene expression datasets \texttt{Arabidopsis} and \texttt{Leukemia} are obtained from \citep{li2010inexact}. Additionally, the NCI 60 dataset \texttt{Staunton} is chosen from \citep{culhane2003cross}. Finally, the yeast eQTL dataset, known as \texttt{realEQTL}, is selected from \citep{zhu2008integrating}.
\end{enumerate}

Our experimental findings unveiled the potential failure of AManPG in cases where the solution of the Lyapunov equation \cite[Equation (15) in Section 3]{huang2022extension} is non-unique or non-existent, even when utilizing the original unmodified code. To ensure a fair and equitable comparison among algorithms, we excluded scenarios where AManPG may encounter such challenges. Additionally, we observed that ARPG frequently terminated prematurely without achieving the desired accuracy due to its stopping criteria.

The results obtained from our experiments with synthetic data are presented in \cref{tab:SPCA} and visualized in \cref{fig:SPCA}. The outcomes obtained from real data are reported in \cref{tab:SPCAeral}. It proved to be challenging to achieve convergence of SOC and ARPG towards the termination condition threshold, whereas our proposed method consistently exhibited convergence properties akin to second-order algorithms. In terms of objective function values, our method demonstrated superior performance in almost all experiments. With the exception of AManPG, which was specifically designed for the SPCA problem, our method generally exhibited the lowest time consumption. These results serve as compelling evidence for the effectiveness of our second-order algorithm, which leverages second-order information, and its superiority over alternative algorithms in tackling the SPCA problem. The experimental results further reinforce the assertion made in the previous section regarding the exceptional performance of our algorithm in similar problem domains.

\section{Conclusion}
In this work, we have proposed a trust region method for minimizing an SC$^1$ function on a Riemannian manifold. We have established our method's global convergence, local convergence near nondegenerate local minima, and superlinear local convergence rate under relaxed smoothness requirement on the objective function and retraction.
We have also provided the first theoretical guarantee of trust region's inactivity near nondegenerate local minima for trust region methods applying to SC$^1$ functions, which is also new in the Euclidean case.
To demonstrate the superiority of our method, we have applied it to solve the subproblem of the augmented Lagrangian method on manifolds and performed extensive numerical experiments. Our results show that our proposed method outperforms existing state-of-the-art methods, achieving better convergence performance.

\begin{table}[ht]
	\caption{~Comparison of SPCA with synthetic data. Bold numbers indicate superior results. The data matrix $A\in \mathbb{R}^{50\times n}$. The tuple $(n, r, \mu) = (2000, 20, 1.00)$, with one of the elements varying.}
	\label{tab:SPCA}%\\
	\centering
	\begin{threeparttable}
		\setlength\tabcolsep{3mm}
		\resizebox{\linewidth}{!}
		% \scriptsize
		{\begin{tabular}{@{}ccccccccc@{}}
				\toprule
				      &      & AManPG         & ARPG     & SOC      & LS-\uppercase\expandafter{\romannumeral1} & LS-\uppercase\expandafter{\romannumeral2} & ALMSRTR           \\ \midrule
				\multicolumn{7}{l}{Running time(s)}                                                                                                                             \\
				$n$   & 500  & 15.02          & 56.77    & 34.25    & 6.82                                      & 6.52                                      & \textbf{5.89}     \\
				      & 1000 & 21.32          & 74.65    & 99.26    & 14.66                                     & 12.83                                     & \textbf{10.80}    \\
				      & 1500 & 30.18          & 86.53    & 212.87   & 26.75                                     & 22.17                                     & \textbf{17.60}    \\
				      & 2000 & 31.28          & 77.78    & 365.03   & 36.24                                     & 29.46                                     & \textbf{27.08}    \\
				      & 2500 & \textbf{36.34} & 84.11    & 535.88   & 51.38                                     & 42.45                                     & 38.08             \\
				      & 3000 & \textbf{42.09} & 79.32    & 768.18   & 67.21                                     & 47.13                                     & 54.81             \\
				$r$   & 5    & \textbf{1.47}  & 16.54    & 226.60   & 14.88                                     & 7.76                                      & 8.79              \\
				      & 10   & \textbf{4.96}  & 8.06     & 298.78   & 25.07                                     & 18.19                                     & 21.18             \\
				      & 15   & \textbf{17.19} & 23.30    & 337.94   & 31.18                                     & 24.61                                     & 22.50             \\
				      & 25   & 45.01          & 163.72   & 405.76   & 51.50                                     & 40.36                                     & \textbf{33.31}    \\
				$\mu$ & 0.25 & \textbf{45.87} & 57.965   & 369.03   & 91.646                                    & 109.44                                    & 70.67             \\
				      & 0.50 & \textbf{33.01} & 42.74    & 369.25   & 52.09                                     & 51.33                                     & 37.94             \\
				      & 0.75 & 38.32          & 59.69    & 369.09   & 34.41                                     & \textbf{32.88}                            & 33.27             \\
				      & 1.25 & 34.75          & 91.48    & 368.12   & 34.28                                     & \textbf{29.67}                            & 37.43             \\
				\multicolumn{7}{l}{Loss function:$\left\{-\mathrm{tr}\left(P^{T} A^{T} A P\right)+\mu\|P\|_{1}\right\}$}                                                        \\
				$n$   & 500  & -347.78        & -347.75  & -346.87  & -347.59                                   & -347.59                                   & \textbf{-349.15}  \\
				      & 1000 & -762.22        & -762.19  & -761.43  & -761.90                                   & -761.90                                   & \textbf{-763.53}  \\
				      & 1500 & -1209.59       & -1209.53 & -1209.42 & -1209.70                                  & -1209.70                                  & \textbf{-1211.15} \\
				      & 2000 & -1621.05       & -1621.37 & -1620.70 & -1621.05                                  & -1621.05                                  & \textbf{-1622.98} \\
				      & 2500 & -2082.34       & -2082.72 & -2080.97 & -2082.82                                  & -2082.82                                  & \textbf{-2084.18} \\
				      & 3000 & -2516.72       & -2516.77 & -2515.30 & -2517.06                                  & -2517.06                                  & \textbf{-2518.69} \\
				$r$   & 5    & -1520.36       & -1520.33 & -1520.32 & \textbf{-1520.47}                         & \textbf{-1520.47}                         & -1520.42          \\
				      & 10   & -1630.61       & -1630.62 & -1630.35 & -1630.36                                  & -1630.36                                  & \textbf{-1631.21} \\
				      & 15   & -1631.77       & -1631.73 & -1631.46 & -1632.13                                  & -1632.13                                  & \textbf{-1633.14} \\
				      & 25   & -1606.38       & -1606.50 & -1605.76 & -1606.85                                  & -1606.85                                  & \textbf{-1607.78} \\
				$\mu$ & 0.25 & -1873.23       & -1873.18 & -1870.19 & -1873.26                                  & -1873.26                                  & \textbf{-1873.39} \\
				      & 0.50 & -1781.25       & -1781.05 & -1779.03 & -1781.27                                  & -1781.27                                  & \textbf{-1781.89} \\
				      & 0.75 & -1700.87       & -1700.10 & -1699.01 & -1700.12                                  & -1700.12                                  & \textbf{-1701.46} \\
				      & 1.25 & -1555.04       & -1555.02 & -1555.26 & -1555.56                                  & 1555.56                                   & \textbf{-1557.73} \\ \bottomrule
			\end{tabular}}
		% \begin{tablenotes}
		% \item Bold numbers imply better results. The data matrix $A\in \R^{50\times n}$. The tuple $(n, r, \mu) = (2000, 20, 1.00)$, with one of the elements varying.

		% \end{tablenotes}
	\end{threeparttable}
\end{table}

\begin{figure}[th]
	\centering
	\subfigbottomskip=2pt
	\subfigcapskip=-5pt
	\subfigure{
		\includegraphics[width=0.32\linewidth]{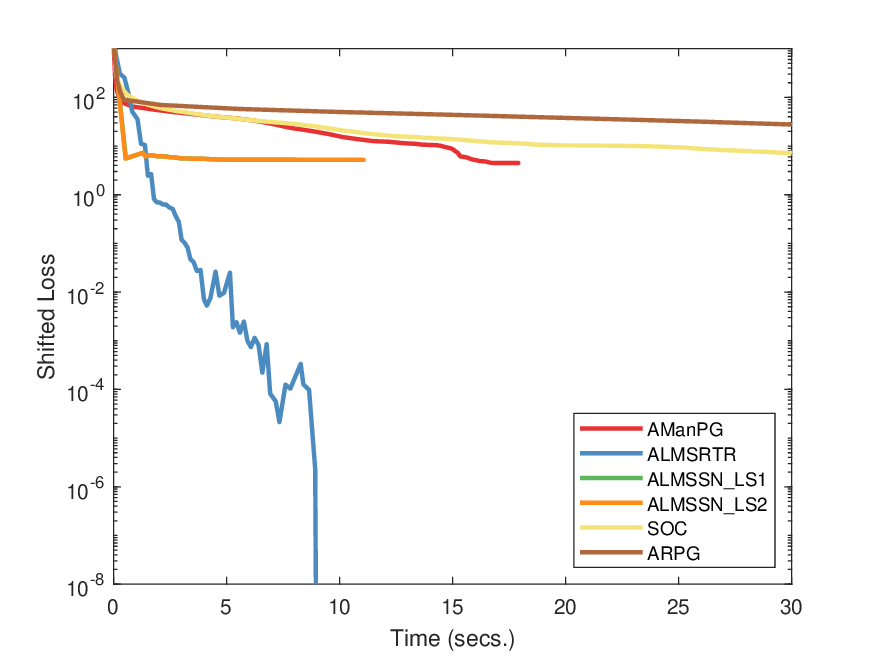}}
	\subfigure{
		\includegraphics[width=0.32\linewidth]{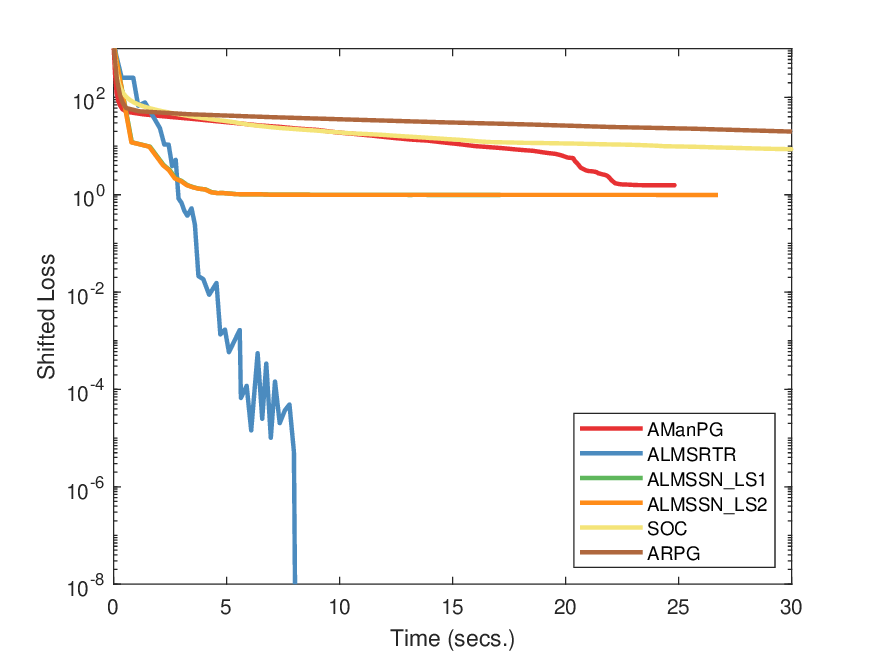}}
	\subfigure{
		\includegraphics[width=0.32\linewidth]{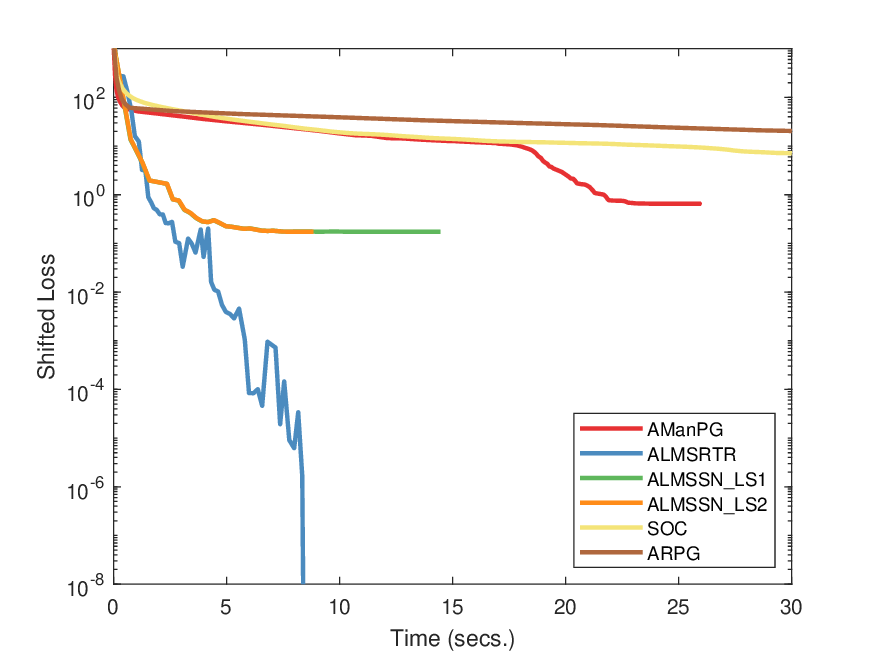}}
	\\
	\subfigure{
		\includegraphics[width=0.32\linewidth]{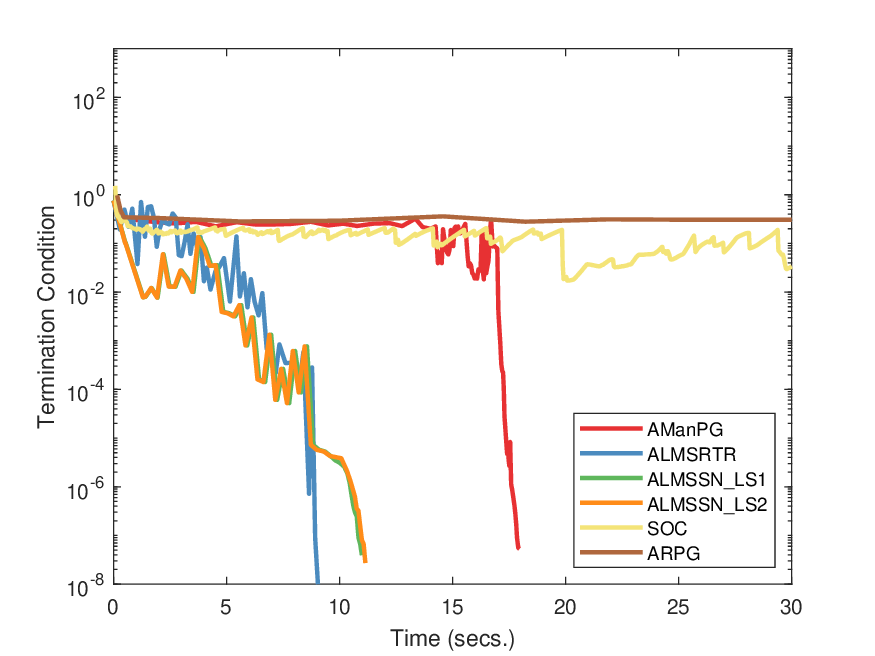}}
	\subfigure{
		\includegraphics[width=0.32\linewidth]{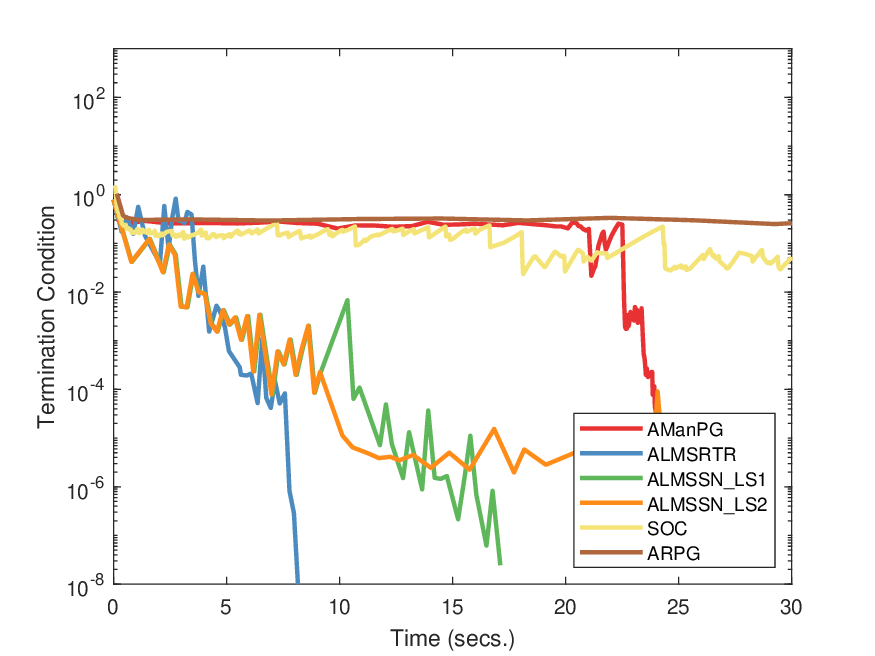}}
	\subfigure{
		\includegraphics[width=0.32\linewidth]{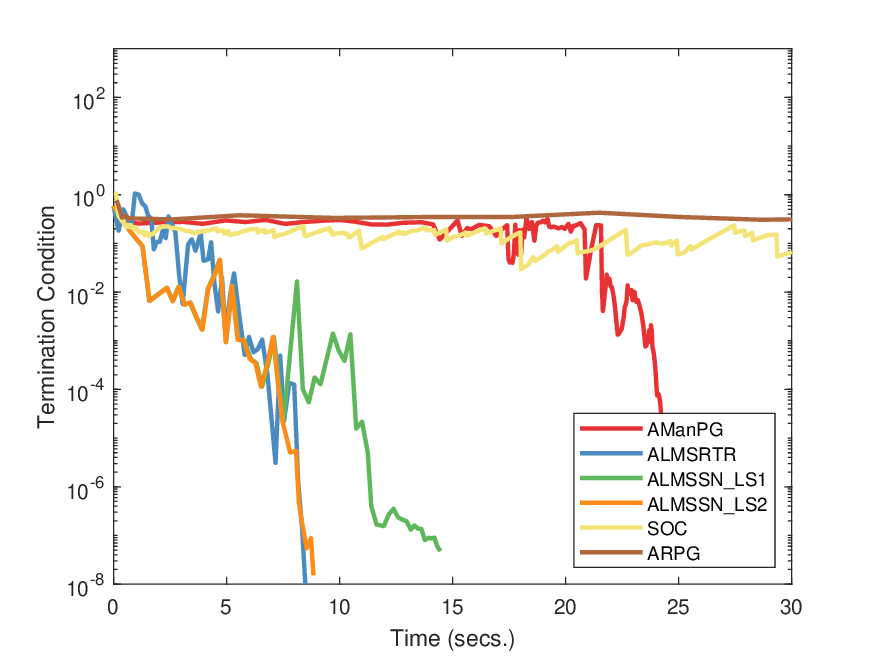}}
	\caption{\ \ Comparisons of SPCA with synthetic data, $(n, r , \mu) = (1000, 20, 1.00)$, and three different $A^TA$. The definitions in this figure are similar to those in Figure \ref{fig:cm}. }
	\label{fig:SPCA}
\end{figure}

\begin{table}[ht]%{c}
	\caption{~Comparison of SPCA with real data. Bold numbers indicate superior results. The slash \texttt{/} indicates that the value of the termination condition at the maximum number of iterations is still greater than $10^{-3}$, indicating that the manifold constraint is not fully satisfied and the loss function value is meaningless. It should be noted that our algorithm ALMSRTR meets the threshold of the termination condition for all settings in this table.}
	\label{tab:SPCAeral}%\\
	\centering
	\begin{threeparttable}
		\setlength\tabcolsep{3mm}
		\resizebox{\linewidth}{!}
		% \scriptsize
		{\begin{tabular}{@{}ccccccccc@{}}
				\toprule
				Dataset                      & $(n,r,\mu)$        & AManPG             & ARPG               & SOC       & LS-\uppercase\expandafter{\romannumeral1} & LS-\uppercase\expandafter{\romannumeral2} & ALMSRTR             \\ \midrule
				\multicolumn{8}{l}{Running time(s)}                                                                                                                                                                                   \\
				\multirow{4}{*}{Arabidopsis} & $(834, 10, 0.50)$  & 14.16              & 28.58              & 34.83     & 37.32                                     & 11.44                                     & \textbf{9.93}       \\
				                             & $(834, 10, 0.25)$  & \textbf{6.05}      & 28.14              & 34.64     & 27.88                                     & 27.22                                     & 7.14                \\
				                             & $(834, 15, 0.50)$  & 22.16              & 36.37              & 45.09     & 30.73                                     & 30.31                                     & \textbf{19.17}      \\
				                             & $(834, 15, 0.25)$  & \textbf{19.50}     & 43.30              & 44.67     & 38.73                                     & 36.21                                     & 23.46               \\
				\multirow{4}{*}{Leukemia}    & $(1255, 10, 0.50)$ & 12.98              & 29.21              & 97.25     & 14.33                                     & \textbf{11.76}                            & 18.83               \\
				                             & $(1255, 10, 0.25)$ & \textbf{8.02}      & 35.22              & 96.83     & 33.36                                     & 57.98                                     & 13.94               \\
				                             & $(1255, 15, 0.50)$ & \textbf{16.18}     & 48.22              & 109.90    & 55.04                                     & 46.51                                     & 21.79               \\
				                             & $(1255, 15, 0.25)$ & 33.65              & 45.83              & 109.13    & 49.12                                     & 64.99                                     & \textbf{21.46}      \\
				\multirow{4}{*}{realEQTL}    & $(1260, 10, 0.50)$ & 37.00              & 35.23              & 96.64     & 40.47                                     & 40.53                                     & \textbf{9.99}       \\
				                             & $(1260, 10, 0.25)$ & 36.37              & \textbf{35.43}     & 97.10     & 64.28                                     & 65.10                                     & 38.91               \\
				                             & $(1260, 15, 0.50)$ & 50.30              & 50.17              & 111.09    & 55.80                                     & 59.05                                     & \textbf{41.13}      \\
				                             & $(1260, 15, 0.25)$ & 49.73              & 49.75              & 111.13    & 136.97                                    & 140.96                                    & \textbf{31.08}      \\
				\multirow{4}{*}{Staunton100} & $(1517, 10, 0.50)$ & \textbf{9.35}      & 37.50              & 150.69    & 10.19                                     & 10.10                                     & 10.15               \\
				                             & $(1517, 10, 0.25)$ & 9.59               & 37.09              & 149.28    & 44.03                                     & 48.41                                     & \textbf{8.78}       \\
				                             & $(1517, 15, 0.50)$ & \textbf{7.85}      & 44.25              & 163.09    & 24.41                                     & 20.19                                     & 40.04               \\
				                             & $(1517, 15, 0.25)$ & \textbf{21.08}     & 51.93              & 162.80    & 29.52                                     & 30.56                                     & 33.25               \\
				\multirow{4}{*}{Staunton200} & $(2455, 10, 0.50)$ & \textbf{16.83}     & 39.68              & 450.72    & 98.19                                     & 24.75                                     & 29.02               \\
				                             & $(2455, 10, 0.25)$ & \textbf{7.99}      & 43.35              & 451.48    & 70.53                                     & 80.20                                     & 37.96               \\
				                             & $(2455, 15, 0.50)$ & \textbf{19.26}     & 61.04              & 461.16    & 62.75                                     & 55.36                                     & 70.18               \\
				                             & $(2455, 15, 0.25)$ & 44.01              & 54.17              & 461.32    & 91.42                                     & 60.66                                     & \textbf{35.24}      \\
				\multicolumn{6}{l}{Loss function:$\left\{-\mathrm{tr}\left(P^{T} A^{T} A P\right)+\mu\|P\|_{1}\right\}$}                                                                                                              \\
				\multirow{4}{*}{Arabidopsis} & $(834, 10, 0.50)$  & \textbf{-65867.82} & -65867.73          & -65864.81 & -65867.42                                 & -65867.42                                 & -65867.51           \\
				                             & $(834, 10, 0.25)$  & -65922.30          & -65922.30          & -65938.12 & -65922.14                                 & -65922.14                                 & \textbf{-65922.31}  \\
				                             & $(834, 15, 0.50)$  & -72553.02          & -72552.31          & -72546.44 & -72552.46                                 & -72552.46                                 & \textbf{-72553.54}  \\
				                             & $(834, 15, 0.25)$  & -72633.64          & -72633.51          & -72629.79 & -72633.59                                 & -72633.59                                 & \textbf{-72634.14}  \\
				\multirow{4}{*}{Leukemia}    & $(1255, 10, 0.50)$ & -53787.19          & \textbf{-53787.27} & -53780.28 & -53787.02                                 & -53787.02                                 & -53787.26           \\
				                             & $(1255, 10, 0.25)$ & -53854.00          & -53854.00          & -53849.77 & -53853.94                                 & -53853.94                                 & \textbf{-53854.03}  \\
				                             & $(1255, 15, 0.50)$ & -60108.81          & -60109.28          & -60101.74 & -60109.03                                 & -60109.03                                 & \textbf{-60109.76}  \\
				                             & $(1255, 15, 0.25)$ & -60208.75          & -60208.81          & -60203.91 & -60208.78                                 & -60208.77                                 & \textbf{-60209.07}  \\
				\multirow{4}{*}{realEQTL}    & $(1260, 10, 0.50)$ & -191911.63         & -191911.63         & /         & -191912.62                                & -191912.62                                & \textbf{-191912.85} \\
				                             & $(1260, 10, 0.25)$ & -191972.26         & -191972.26         & /         & -191972.76                                & -191972.76                                & \textbf{-191972.88} \\
				                             & $(1260, 15, 0.50)$ & -207390.14         & -207390.16         & /         & -207390.17                                & -207390.17                                & \textbf{-207390.61} \\
				                             & $(1260, 15, 0.25)$ & -207477.62         & -207477.61         & /         & -207477.65                                & -207477.65                                & \textbf{-207477.87} \\
				\multirow{4}{*}{Staunton100} & $(1517, 10, 0.50)$ & -37081.99          & -37080.75          & -37080.17 & -37081.82                                 & -37081.82                                 & \textbf{-37082.06}  \\
				                             & $(1517, 10, 0.25)$ & -37147.01          & -37147.02          & -37146.08 & -37147.54                                 & -37147.54                                 & \textbf{-37147.66}  \\
				                             & $(1517, 15, 0.50)$ & -42827.91          & -42827.66          & -42822.76 & -42827.88                                 & -42827.88                                 & \textbf{-42828.22}  \\
				                             & $(1517, 15, 0.25)$ & \textbf{-42926.71} & -42926.59          & -42923.03 & -42926.59                                 & -42926.59                                 & -42926.68           \\
				\multirow{4}{*}{Staunton200} & $(2455, 10, 0.50)$ & -43880.94          & -43880.94          & -43878.25 & -43880.78                                 & -43880.78                                 & \textbf{-43880.98}  \\
				                             & $(2455, 10, 0.25)$ & \textbf{-43962.73} & \textbf{-43962.73} & -43960.11 & -43962.68                                 & -43962.68                                 & \textbf{-43962.73}  \\
				                             & $(2455, 15, 0.50)$ & -50790.16          & -50790.20          & -50779.81 & -50790.27                                 & -50790.27                                 & \textbf{-50790.29}  \\
				                             & $(2455, 15, 0.25)$ & -50912.84          & -50912.86          & -50905.02 & -50912.65                                 & -50912.65                                 & \textbf{-50912.93}  \\ \bottomrule
			\end{tabular}}
		% \begin{tablenotes}
		% \item Bold numbers imply better results. The data matrix $A\in \R^{50\times n}$. The tuple $(n, r, \mu) = (2000, 20, 1.00)$, with one of the elements varying.

		% \end{tablenotes}
	\end{threeparttable}
\end{table}
%

%\clearpage
% \bibliographystyle{spmpsci}
\bibliography{rtrssn.bib}
\end{document}